\documentclass[microtype]{gtpart}
\usepackage{mathdots}
\usepackage[all]{xy}
\usepackage{pdfsync}
\title{The $\G_m$--equivariant Motivic Cohomology of Stiefel Varieties}

\author{Ben Williams}
\givenname{Ben}
\surname{Williams}
\address{ USC\\
 Department of Mathematics\\
 Kaprielian Hall\\
 3620 South Vermont Avenue\\
Los Angeles, CA 90089-2532\\
USA.
}
\email{tbwillia@usc.edu}
\urladdr{}

\keyword{Equivariant}
\keyword{Motivic cohomology}
\keyword{Chow group}
\keyword{Fiber-to-base}
\keyword{Stiefel}
\keyword{Projective general linear group}

\subject{primary}{msc2000}{19E15}
\subject{secondary}{msc2000}{14C15}
\subject{secondary}{msc2000}{18G55}

\arxivreference{}
\arxivpassword{}

\volumenumber{}
\issuenumber{}
\publicationyear{}
\papernumber{}
\startpage{}
\endpage{}
\doi{}
\MR{}
\Zbl{}
\received{}
\revised{}
\accepted{}
\published{}
\publishedonline{}
\proposed{}
\seconded{}
\corresponding{}
\editor{}
\version{}

\newcommand{\id}{\mathrm{id}}
\newcommand{\isom}{\cong}

\newcommand{\mto}[1]{\stackrel{#1}{\longrightarrow}}
\newcommand{\isomto}{\mto{\isom}}

\newcommand{\sm}{\setminus}

\newtheorem{thm}{Theorem}
\newtheorem{lemma}[thm]{Lemma}
\newtheorem{theorem}[thm]{Theorem}
\newtheorem{prop}[thm]{Proposition}
\newtheorem{cor}{Corollary}[thm]

\newenvironment{dfn}{\null\par \textbf{Definition:}}{ \null \par}

\newcommand{\dfnnd}[1]{\textit{#1}}

\newcommand{\Tor}{\operatorname{Tor}}
\newcommand{\Ext}{\operatorname{Ext}}

\newcommand{\EA}{\Lambda}

\newcommand{\Sl}{\operatorname{SL}}

\newcommand{\Gl}{\operatorname{GL}}

\newcommand{\diag}[1]{\operatorname{diag}\left({#1}\right)}

\renewcommand{\d}[1]{\widehat{#1}}

\newcommand{\Hom}{\operatorname{Hom}}
\newcommand{\colim}{\operatornamewithlimits{colim}}
\newcommand{\cat}[1]{\text{\bf #1}}

\newcommand{\Pre}{\cat{Pre}}
\newcommand{\tensor}{\otimes}

\newcommand{\spec}{\operatorname{Spec}}

\newcommand{\Spec}{\spec}

\newcommand{\Sm}{\cat{Sm}}

\newcommand{\pt}{\text{\rm pt}}

\newcommand{\Nis}{\text{\it Nis}}

\newcommand{\et}{\text{\'et}}

\newcommand{\A}{\mathbb{A}}
\renewcommand{\P}{\mathbb{P}}

\renewcommand{\sh}[1]{\mathcal{#1}}

\newcommand{\G}{\mathbb{G}}

\newcommand{\hocolim}{\operatornamewithlimits{hocolim}}

\newcommand{\M}{\mathbb{M}}

\newcommand{\weq}{\simeq}

\providecommand{\deg}{\operatorname{deg}}

\newcommand{\Spaces}{\cat{Spaces}}
\newcommand{\sk}{\operatorname{sk}}

\newcommand{\Rlim}{\operatornamewithlimits{Rlim}}
\renewcommand{\d}[1]{\hat{#1}}
\newcommand{\tchh}[1]{\operatorname{tch}{#1}}
\newcommand{\op}{\text{op}}
\newcommand{\pre}{\text{pre}}
\newcommand{\cqld}[6]{ {#1} \ar^{#4}@<0.5ex>[rr] \ar_{#5}@<-0.5ex>[rr]  & & {#2} \ar^{#6}[rr] & & {#3}}

\newcommand{\dfnd}[1]{\textit{#1\/}}
\newcommand{\GL}{\Gl}

\providecommand{\spec}{\operatorname{Spec}}

\providecommand{\Spec}{\spec}

\newcommand{\Aone}{{\mathbb{A}^1}}

\newcommand{\apex}{\operatorname{vex}}
\newcommand{\PGL}{\operatorname{PGL}}
\renewcommand{\vec}[1]{\mathbf{#1}}
\newcommand{\Ann}{\operatorname{Ann}}
\newcommand{\CH}{\mathrm{CH}}
\newcommand{\LL}{\mathbb{L}}

\newcommand{\Sh}{\cat{Sh}}
\newcommand{\smsh}{\wedge}
\newcommand{\cofib}{\operatornamewithlimits{cofib}}
\newcommand{\Eoh}{\mathrm{E}}
\renewcommand{\Spaces}{s\cat{Sh}_\tau(\cat{C})}
\newcommand{\nSpaces}{s\cat{Sh}_\Nis(\Sm/k)}

\newcommand{\Hoh}{\mathrm{H}}

\numberwithin{thm}{section}

\begin{document}

\begin{abstract}
 
\noindent
We derive a version of the Rothenberg--Steenrod, fiber-to-base, spectral sequence for cohomology
theories represented in model categories of simplicial presheaves. We then apply this spectral
sequence to calculate the equivariant motivic cohomology of $\GL_n$ with a general $\G_m$--action, this
coincides with the equivariant higher Chow groups. The motivic cohomology of $\PGL_n$ and some of the equivariant motivic cohomology of a
Stiefel variety, $V_m(\A^n)$, with a general $\G_m$--action is deduced as a corollary.

\paragraph{Mathematics Subject Classification 2000}

\end{abstract}
\maketitle


\section*{Introduction}
Let $k$ be a field, not necessarily algebraically closed or of characteristic $0$. Let $G$ be a linear algebraic group acting on a variety
$X$. The equivariant higher Chow groups of a variety $X$ are defined in general in \cite{EG}, developing an idea presented in
\cite{TOTAROCLASS} where the classifying space of $G$ is constructed. Subsequent calculations of equivariant Chow groups have been carried
out especially in the important cases of $A_G^*$, as in \cite{VISTOLIPGL}, or in order to calculate the ordinary Chow groups of the moduli
spaces $\mathcal{M}_g$ of genus $g$ curves, \cite[Appendix by A.~Vistoli]{EG}. The chief tool used in these calculations has been the
equivariant stratification of varieties, and they are algebreo-geometric in nature. The higher equivariant Chow groups have not been much
computed.

In this paper we adopt the view that the equivariant higher Chow groups should behave much like Borel equivariant singular cohomology
groups, at least when the groups and varieties concerned are without arithmetic complications, and therefore methods of $\A^1$--homotopy may be
employed. There are a great many restrictions on $X$ and $G$ that limit the applicability of this idea, but in the special case
where $X$ is a Stiefel variety, $V_m(\A^n)$, parametrizing full-rank $n \times m$--matrices and where $G = \G_m$, or more generally $G$ is a
torus, the restrictive criteria are all satisfied, and we may deduce some of $\CH^*_G(X, *)_F$ (where $F$ is a field of coefficients) from $\CH^*(X,
*)_F$. This is done in corollary \ref{CORLAST}.

We may identify the Stiefel variety $V_m(\A^n)$ as the space of matrices representing injective linear maps $A\co k^m \to k^n$ where both source and target are
equipped with a standard basis. By assigning the basis elements $\Z$--gradings, say $u_1, \dots, u_m$ and $v_1, \dots , v_n$ to be specific,
we endow the source and target with the structure of a $\G_m$--representation. Consequently, if we let $\G_m$ act on $V_m(\A^n)$ by
\[ z \cdot A = \begin{pmatrix} z^{-u_1} & \\ & z^{-u_2} & \\ & & \ddots \\ & & & z^{-u_m} \end{pmatrix} A \begin{pmatrix} z^{v_1} & \\ &
  z^{v_2} & \\ & & \ddots \\ & & & z^{v_n} \end{pmatrix}, \] we incorporate the $\Z$--grading into the structure of $V_m(\A^n)$ in the
following sense: if $R$ is a $\Z$--graded $k$--algebra, then $\Spec R$ is a $k$--scheme equipped with a $\G_m$---action, and a $\G_m$--equivariant
map $\Spec R \to V_m(\A^n)$ classifies a graded, full-rank map $R^m \to R^n$. In the special case where $n=m$, where the source of the linear map is entirely in degree $0$ and the target entirely in degree
$1$, we recover as a quotient the scheme $\PGL_n$, and we present the higher Chow groups of $\PGL_n$ in corollary \ref{c:PGLn}.

The structure of the paper is as follows. The first section is written in general language, applying to an unspecified cohomology theory
satisfying certain properties and representable in some model structure on a category of simplicial
presheaves on a category $\cat{C}$. In proposition \ref{p:resent} a spectral sequence  is used to compute the cohomology of a bar complex $B(G, X)$, which is a version of the Borel
equivariant cohomology of $X$ with respect to the $G$--action. If the cohomology of $G$ and $X$ are particularly well-behaved, this spectral
sequence admits description on the $\Eoh_2$--page, as described in theorem \ref{p:MRSSS}.

In theorem \ref{p:mainres}, the case of a free action is treated, describing a quotient in $\cat{C}$, say $X = Y/G$, in terms of the
cohomology of $Y$ and $G$, provided a battery of conditions on $Y$ and $G$ are satisfied, and provided the quotient $X = Y/G$ satisfies a
local triviality condition. This is an analogue of a spectral sequence in classical algebraic topology that goes by several names, among them
``Rothenberg-Steenrod'' and ``fiber-to-base Eilenberg--Moore''; it appears in \cite{ROTHENBERGSTEENROD}.

The second section then specializes to the case of motivic cohomology, and to our form of Borel equivariant motivic cohomology. We argue in
proposition \ref{p:smoothGroup} that under the hypotheses that are satisfied throughout the paper, chiefly that $G$ be a special group, that this Borel
equivariant motivic cohomology is isomorphic after re-indexing to the equivariant higher Chow
groups of \cite{EG}. We employ the motivic-cohomological indexing throughout. Some facts on the ordinary motivic cohomology of $\GL_n$ and related varieties are recalled,
before we establish a sequence of technical results, always arguing about differentials in spectral sequences, culminating in a
nearly-complete description of a convergent spectral sequence calculating the $\G_m$--equivariant motivic cohomology of $\GL_n$ for a general
$\G_m$--action on both the left and right, in proposition \ref{p:lrsss}. In the final section, we deduce what we can of the analogous spectral sequence
for a general action of $\G_m$ on $V_m(\A^n)$. Some technicalities on homological algebra, which are required for the description of the
$\Eoh_2$--pages of the spectral sequences, are included in an appendix.

\section{Spectral Sequences for Bisimplicial Sheaves}
\subsection{Spaces}

In the first section we work in some generality. We suppose $\cat{C}$ is a site, with topology $\tau$. As a technical convenience we assume
$\cat{C}$ has enough points. We think of the category of simplicial sheaves, $\Spaces$, on $\cat{C}$ as being a category of spaces, and it
is here that we carry out the bulk of our homotopy theory.

There are several model structures on the categories $\Spaces$. Among them are various global model structures, where weak equivalences are
detected objectwise, and various local model structures, where the weak equivalences $\sh{F} \to \sh{H}$ are the maps of simplicial
presheaves that yield weak equivalences at all points $p^*$ of the site $\cat{C}$. We choose as our model structures of preference the
\dfnnd{injective} model structures where the cofibrations are the maps $\sh{F} \to \sh{H}$ which are monomorphisms of simplicial
sheaves. The fibrations are thus determined as the maps that satisfy a lifting property with respect to trivial cofibrations. By the
\textit{global model structure} or the \textit{local model structure} we mean the global injective and local injective model structure
respectively. The injective local model structure originates with Joyal, and is presented in \cite{JARSIMPPRE}.

We endow $\Spaces$ with a model structure, $M$, which is a left Bousfield localization of the injective local
model structure. The best-known examples of such model structures are the injective model structures themselves,
especially for the \'etale topology, and the $\A^1$--model structure of \cite{MV}. In a sop to excessive generality, we do not require that
representable presheaves be sheaves on $\cat{C}$. We shall write `global weak equivalence', `injective weak equivalence' in the global or
local cases; the bare term `weak equivalence', when applied to sheaves or presheaves, will mean `$M$--weak-equivalence'.

We remark that in addition to there being an injective local model structure on $s \Sh(\cat{C})$, there is an injective $\tau$--local model structure on $s
\Pre(\cat{C})$, where the weak equivalences are again detected on stalks, so that the natural transformation from the identity to the
associated-sheaf functor $F \mapsto a(F)$ is always a weak equivalence. The distinction between the injective model structure on $\Spaces$
and that defined on $s\Pre(\cat{C})$ is nugatory.

We observe that the local injective structure on $\Spaces$ is in fact a localization of the global structure, see for instance
\cite{FLASQUE}. In particular, maps which are weak equivalences in the global model category will be weak equivalences in all the model
categories under consideration.

\begin{sloppypar}
Both $\Spaces$ and $s \Pre(\cat{C})$, with any one of the previously mentioned model structures, have associated pointed structures,
the underlying categories of which will be denoted by $\Spaces_+$ and $s \Pre(\cat{C})_+$.
\end{sloppypar}

\subsection{Cohomology}
We fix a bigraded representable cohomology theory $L^{*,*}$, which is to say a bigraded family of fibrant objects $L(p,q) \in \Spaces_+$ such
that 
\begin{equation*}
  L^{p,q}(X) = [ X_+, L(p,q)]
\end{equation*}
in $s \Sh(\cat{C})_+$. We demand that there be $M$--weak-equivalences $L(p,q) \to \Omega L(p+1,q)$, so that there is a suspension
isomorphism $L^{p+r, q}(\Sigma^r X) \isom L^{p,q}(X)$, from which it follows that $L^{p,q}(\cdot)$ is an abelian-group valued functor. We
require finally that there be multiplication maps $L(p,q) \wedge L(p',q') \to L(p+p', q+q')$ satisfying the usual diagrammatic conditions,
endowing $L^{p,q}(X)$ with a functorial ring structure, which we require to be graded-commutative in the first grading and
strictly-commutative in the second.

We are essentially demanding that $L(\cdot, q)$ be a ring object in a category of graded $S^1$--spectra, and we might have presented the above
requirements in the context of some graded stable model category. We choose not to do so, since the cost of having to pass from unstable
to stable model structures outweighs the benefit of streamlined arguments in the stable setting.

If $X$ is an object in $\Spaces_+$, we write $\tilde L^{p,q} = [X, L(p,q)]$.

Theories meeting the criteria we demand of $L(p,q)$ abound. The application of the general theory in
the sequel will be to the case of $L(p,q) = \mathbb{H}R(p,q)$, the motivic Eilenberg-MacLane spaces
representing $R$--valued motivic cohomology, see \cite{MOTEMSP} for instance, in the $\A^1$--model
structure on $s \Sh_\Nis(\Sm/k)$. Other examples include, for various choices of $\cat{C}$ and $M$,
algebraic $K$--theory and \'etale cohomology.

Let $\mathcal{A}$ be a collection of objects of $\Spaces$. We say $L^{*,*}(\cdot)$ is  \dfnnd{bounded below on $\mathcal{A}$} if it meets the following condition: for all integers $q$, there exists some $c(q)$ such that
$L^{p,q}(X) =0 $ for all $X \in \mathcal{A}$ whenever $p < c(q)$. If $\mathcal{A} = \Spaces$, we simply say $L^{*,*}$ is \dfnnd{bounded
  below}. Unfortunately, the boundedness of motivic cohomology is
essentially the Beilinson-Soul\'e vanishing conjecture, which is currently known only in some cases---for instance for $\Hoh^{*,*}(\Spec
F;\Z)$ where $F$ is either a finite field or a number field.

\subsection{Bisimplicial Sheaves and Cohomology}

Our references for simplicial homotopy theory are \cite{HIRSCHHORN}, especially chapter 18,
and \cite{BOUSKAN}. Let $X_\bullet$ be a simplicial object in $\Spaces$ or $s \Pre(\cat{C})$, which one views either as a bisimplicial
\mbox{(pre-)}sheaf in $\Delta^{\op} \times \Delta^{\op} \to \cat{Sh}(\cat{C})$ (in $\Delta^{\op} \times \Delta^{\op} \to \cat{Pre}(\cat{C})$
respectively), or as a $\Delta^{\op}$--shaped diagram in the category $\Spaces$ (or in the category $s \Pre(\cat{C})$ respectively). There
are a diversity of methods of `realizing' a simplicial object $X_\bullet$. In our model structures, where all objects are cofibrant, these
result in identical, isomorphic, or at worst weakly-equivalent objects.

A simplicial object $X_\bullet$ in $\Spaces$ (or $s \Pre(\cat{C})$) is a $\Delta^\op$--shaped diagram 
\[X_\bullet\co \Delta^\op \to \Spaces\] (or
$X_\bullet\co \Delta^\op \to s \Pre(\cat{C})$). We define the realization $|X_\bullet|$ to be $\hocolim_{\Delta^\op} X_\bullet$, using the
explicit construction of $\hocolim$ as given in \cite{HIRSCHHORN}.

From the description given there it is clear that if $X_\bullet$ is a simplicial object in $\Spaces$, then it does not matter whether the
realization or the homotopy colimit is taken in the category $\Spaces$ or $s \Pre(\cat{C})$, the answer in either case is the same. This
applies indeed to arbitrary homotopy colimits of diagrams in the category $\Spaces$.

The $M$--model structure on $\Spaces$ inherits the cofibrations of the injective model structure, which are simply the monomorphisms of
presheaves and which coincide with the monomorphisms of sheaves. It follows that if $\mathcal{D}$ is a diagram in $\Spaces$, then
$\hocolim \mathcal{D}$ has the same construction for the injective and the $M$--model structure.

We recall from \cite{MV} that a point of a site is a covariant functor that commutes with finite limits and all colimits.  A map $f\co X \to
Y$ in $\Spaces$ or in $s \Pre(\cat{C})$ is an injective weak equivalence if and only if it induces a weak equivalence at all points.  Let
$\cat{I}$ be a small category and let $X$ be an $\cat{I}$--shaped diagram in $\Spaces$ (respectively in $s\Pre(\cat{C})$). By construction of
$\hocolim$, one has 
\begin{equation*} 
  p\big(\hocolim_\cat{I} X\big) = \hocolim_\cat{I} pX.
\end{equation*}
This is given in \cite{MV} as lemma 2.1.20.

One can see directly from the construction of the functor $|\cdot|$ that it preserves colimits. If $X_\bullet$ and $Y_\bullet$ are two
objects in $s\Spaces$, then the projections $\pi\co X_\bullet \times Y_\bullet \to X_\bullet$, along with functoriality, imply there is a weak-equivalence
$|X_\bullet \times Y_\bullet| \to |X_\bullet| \times |Y_\bullet|$. In the pointed case, the preservation of colimits implies that there is an
induced map $|X_\bullet \smsh Y_\bullet| \to |X_\bullet| \smsh |Y_\bullet|$.

Although one can impose a model structure on $s\Spaces$, we do not do so. If $D\co \mathcal{I} \to s\Spaces$ is a diagram (a functor from a
small category) then we denote by $\hocolim_{I} D(i)$ the termwise homotopy-colimit, that is, the object in $s \Spaces$ given by the functor
$n \mapsto \hocolim_{i} D(i)_n$.

\begin{prop} \label{c:1} Let $D\co \mathcal{I} \to s\Spaces$ be a diagram in the category of simplicial spaces. One has $M$--weak-equivalences
  \begin{equation*}
    |\hocolim_{i \in \cat{I}} D(i)| \weq \hocolim_{i \in \cat{I}} |D(i)|.
  \end{equation*}
\end{prop}
\begin{proof}
  We will show these are injective weak equivalences, \textit{a fortiori\/} that they are $M$--weak-equivalences. The given equivalence amounts to
  \begin{equation*}
    \hocolim_{\Delta^{\op}} \hocolim_{i \in \cat{I}} D(i) \weq_{\text{injective}} \hocolim_{i \in \cat{I}},
    \hocolim_{\Delta^{\op}} D(i)
  \end{equation*}
  which is a consequence of the Fubini theorem for $\hocolim$, \cite[Ch. XIII, 3.3]{BOUSKAN}, proved there
  for simplicial sets, which can be promoted to the current setting by arguing at points.
\end{proof}

The following is a special case of \cite[Theorem 18.5.3]{HIRSCHHORN}.

\begin{prop} \label{p:hirsch185} If $\cat{I}$ is a small category and $f\co X \to Y$ is a natural transformation of $\cat{I}$--shaped diagrams in
  $\Spaces$ which is an objectwise $M$--weak-equivalence then the induced map $\hocolim_\cat{I} X \to
  \hocolim_\cat{I} Y$ is an $M$--weak-equivalence. In particular, if $f\co X_\bullet \to Y_\bullet$ is an objectwise $M$--weak-equivalence
 of simplicial objects in $\Spaces$, then $|X_\bullet| \to |Y_\bullet|$ is an $M$--weak-equivalence.
\end{prop}

We denote by $\pt$ the constant sheaf with value the one-point simplicial set. We denote the $L$--cohomology of $\pt$ by $\LL =
L^{*,*}(\pt)$; it is a bigraded ring. The $L$--cohomology of any space is a module over $\LL$.

Suppose $A_\bullet$ is a cosimplicial abelian group. There are two associated chain complexes of abelian
groups, the first of which is the na\"ive $A_*$, in which the terms are precisely the $A_p$ and the differentials are alternating sums of
the coface maps, the second is the normalized chain complex $N_*(A_\bullet)$, where $N_p(A_\bullet)$ is the subgroup of nondegenerate cochains
in $A_p$. It is well-known that the inclusion $N_*(A_\bullet) \to A_*$ is a quasi-isomorphism.

The following proposition allows us to compute the $L$--cohomology of the realization of a simplicial object in $\Spaces$ provided issues of
convergence can be resolved.

\begin{prop} \label{p:BasicSS}
  Suppose $X_\bullet$ is a simplicial object in $\Spaces$. Application of $L^{*,*}(\cdot)$ produces a cosimplicial bigraded abelian
  group. Write $N_*(L^{*,*}(X_\bullet))$ for the associated normalized cochain complex.

  There is a trigraded spectral sequence 
  \[ \Eoh_1^{p,q,*} = N_p(L^{q,*}(X_\bullet)) \Longrightarrow \lim_{s \to \infty} L^{p+q,*}(|\sk_s A_\bullet|), \]
  converging conditionally. The spectral sequence is functorial in
  both $X_\bullet$ and $L^{*,*}$. The differentials on the $r$--th page take the form $d_r\co \Eoh_r^{p,q} \to \Eoh_r^{p+r, q-r+1}$. Additionally,
  the differential on the $\Eoh_1$--page coincides with the differential on $N_*(L^{*,*}(X_\bullet))$, so that the $\Eoh_2$--page may be identified
  with the homology of the complex $L^{*,*}(X_{q-1}) \to L^{*,*}(X_q) \to L^{*,*}(X_{q+1})$.

  Suppose now that $X_\bullet$ and $Y_\bullet$ are two simplicial objects in $\Spaces$. Then there is a natural pairing of spectral
  sequences inducing products \[\Eoh_r^{p,q}(X_\bullet) \tensor \Eoh_r^{p',q'}(Y_\bullet) \to \Eoh^{p+p', q+q'}_r(X _\bullet\times Y_\bullet)\] with
  respect to which the differentials are derivations.
\end{prop}
\begin{proof}
  The third grading does not play much part in this proof, and we suppress it.

  If $X_\bullet$ is an object in $s\Spaces$, then $L^{*,*}(X_\bullet)$ will be used to denote the cosimplicial object $n \mapsto
  L^{*,*}(X_n)$.


  The bulk of our proof amounts to little more than the observation that the argument of \cite[proposition 5.1]{SEGAL} carries over to the
  present setting. Where \cite{SEGAL} has pairs $(X,Y)$, we have cofibers $\cofib(Y \to X)$; the other notational discrepancies between that
  paper and this one are minor.

  The simplicial object $X_\bullet$ is filtered by the skeleta, $\sk_i(X_\bullet)$. Write $B_i$ for the cofiber
  \begin{equation*}
    \xymatrix{ |\sk_{i-1}(X_\bullet)| \ar^{\iota_i}[r] & |\sk_i(X_\bullet)| \ar[r] & B_i }.
  \end{equation*}
  Since $\iota_i$ is an inclusion, it follows that $B_i=|\sk_i(X_\bullet)|/|\sk_{i-1}(X_\bullet)|$. There is a global weak equivalence
  \begin{equation*}
    B_i \weq_{\text{global}} \Sigma^i(\cofib(X_i^{\text{deg}}\to X_i))
  \end{equation*}
  where $X_i^{\text{deg}}$ denotes the image of the degeneracies in $X^i$.

  The filtration by skeleta leads, as in \cite{BOARDMAN}, to an unrolled exact couple:
  \begin{equation*}
    \xymatrix@C=1em{ \ar[r] &  L^{*,*}(|\sk_{-s+1}(X_\bullet)|) \ar[rr] & & L^{*,*}(|\sk_{-s}(X_\bullet)|)
      \ar[dl] \ar[r] & \\ & & L^{*,*}(B_{-s+1}) \ar[ul] }
  \end{equation*}
  This gives rise to a spectral sequence. Using the indexing $\Eoh_1^{p,q} = L^{q,*}(B_{-p})$, the construction yields
  differentials of the form \[d_r\co \Eoh_r^{p,q} \to \Eoh_r^{p+r, q-r+1}\] as claimed.

  The arguments of \cite{SEGAL} apply, in particular the identification of the $\Eoh_1$--page, including the differentials, with the
  reduced chain complex associated to the cosimplicial abelian group $L^{*,*}(X_\bullet)$.

  We note that $L^{*,*}(|\sk_{-s}(X_\bullet)|)= 0$ for $s > 0$, so that
  \begin{equation*}
    \lim_s L^{*}(|\sk_{-s}(X_\bullet)|) = \Rlim_s L^{*}(|\sk_{-s}(X_\bullet)|) = 0,
  \end{equation*}
  and in the terminology of \cite{BOARDMAN}, we immediately have conditional convergence of the spectral sequence to
  \[ \lim_{s \to -\infty} L^{*,*}(|\sk_{-s} (X_\bullet)|) = L^{*,*}(|\sk_{-s}(X_\bullet)|).\]

  Suppose $X_\bullet$, $Y_\bullet$ are two objects in $s\Spaces$. For any
  nonnegative integers $s, s'$, there are Eilenberg-Zilber maps $\sk_{s+s'} (X_\bullet \times Y_\bullet) \to \sk_s(X_\bullet) \times
  \sk_{s'}(Y_\bullet)$; we combine the realization of this map with $|\sk_s(X_\bullet) \times \sk_{s'}(Y_\bullet)| \to |\sk_s(X_\bullet)|
  \times |\sk_{s'}(Y_\bullet)|$ and so obtain a pairing \[\tilde{L}^{*,*}(|\sk_s X_\bullet|) \tensor \tilde{L}^{*,*}(|\sk_{s'} Y_\bullet|)
  \to \tilde{L}^{*,*}(|\sk_{s+s'}(X_\bullet \times Y_\bullet)|).\]

  This product gives rise in the usual way to a pairing on the $\Eoh_2$--page, and it is standard that all differentials $d_r\co \Eoh_r \to
  \Eoh_r$ are derivations with respect to this product, so that the product persists to the $\Eoh_\infty$--page.
\end{proof}

\begin{cor}\label{c:BasicSS}
 Suppose $L^{*,*}(\cdot)$ is bounded below on the set $\{X_s\}_{s=0}^\infty$. The spectral sequence of
  the previous proposition converges to $L^{p+q, *}(|X_\bullet|)$. In each weight $r$, one has
  \[ \Eoh_1^{p,q,r} = N_p(L^{q,r}(X_\bullet)) \Longrightarrow L^{p+q,r}(|X_\bullet|). \]
\end{cor}

\begin{proof}
  Write $\iota_s$ for the map $\iota_s\co |\sk_{-s}(A_\bullet)| \to |\sk_{-s+1}(A_\bullet)|$. Since
  $\iota_n$ is a cofibration for all $n$, one has
  \begin{equation*}
    |A_\bullet| = \colim_s |A_s| = \hocolim_s |A_s| = \operatorname{tel} |A_s|,
  \end{equation*}
  where $\operatorname{tel}$ denotes the mapping telescope construction. A standard argument, as in
  \cite[Chapter 19.4]{MAYCONCISE}, gives
  \begin{equation*}
    {L}^{*,*}(|A_\bullet|) \isom \lim_s {L}^{*,*}(|\sk_{-s}(A_\bullet)|),
  \end{equation*}
  contingent on the vanishing of the $\Rlim$--term in
  \begin{eqnarray*}
    \xymatrix{  0 \ar[r] & \lim_s {L}^{*,*}(|\sk_{-s}(A_\bullet)|) \ar[r] &
 \prod_{s=0}^{\infty}{L}^{*,*}(|\sk_{s}(A_\bullet)|) \ar^{ \hspace{5em}\id - \iota^*}[r] & } \\
\xymatrix{ \prod_{s=0}^\infty {L}^{*,*}(|\sk_{s}(A_\bullet)|) \ar[r] &
\Rlim_s {L}^{*,*}(|\sk_{-s}(A_\bullet)|) \ar[r] & 0.}
  \end{eqnarray*}
  By restricting to particular bidegrees, ${L}^{p,q}(\cdot)$, we obtain
  \begin{eqnarray*}
    \xymatrix{ 0 \ar[r] & \lim_s {L}^{p,q}(|\sk_{-s}(A_\bullet)|) \ar[r] &
  \prod_{s=0}^{\infty} {L}^{p,q}(|\sk_{s}(A_\bullet)|) \ar^{ \hspace{5em}\id - \iota^*}[r] &} \\
 \xymatrix{ \prod_{s=0}^\infty {L}^{p,q}(|\sk_{s}(A_\bullet)|) \ar[r] & 
 \Rlim_s {L}^{p,q}(|\sk_{-s}(A_\bullet)|) \ar[r] & 0,}
  \end{eqnarray*}
  but since $\iota_i^*\co {L}^{p,q}(|\sk_{i}(A_\bullet)|) \to {L}^{p,q}(|\sk_{i-1}(A_\bullet)|)$ is an
  isomorphism when $i>p-c(q)$, it follows that the derived limit $\Rlim_s
  {L}^{p,*}(|\sk_{-s}(A_\bullet)|;R)$ vanishes. Consequently
   \begin{equation*}
    {L}^{*,*}(|A_\bullet|) \isom \lim_s {L}^{*,*}(|\sk_{-s}(A_\bullet)|)
  \end{equation*}
   as required.
   
   The convergence is strong: since $\Eoh_1^{p,q,r} = 0$ if $q< c(r)$, each group can support at most finitely-many differentials.
\end{proof}

\subsection{The Bar Construction}

Suppose $G$ is a group object in $\Spaces$ and $X,Y \in \Spaces$ admit right- and left--actions by $G$ respectively. We can form the
two-sided bar construction $B(X,G,Y)_\bullet$ by precisely the formulas of \cite{MAYCLASS}. It is a simplicial object in $\Spaces$, i.e.~a
bisimplicial sheaf, but we suppress the simplicial indices arising from the intrinsic structure of objects in $\Spaces$. One has
\begin{equation*}
	B(X,G,Y)_n = X \times \overbrace{G \times \dots \times G}^{\text{$n$--times}} \times Y.
\end{equation*}
Bar constructions exhibit a wealth of desirable properties in ordinary homotopy-theory, for which see \cite{MAYCLASS},
and some of these results also hold in the context of the homotopy of sheaves. We shall frequently
prefer to work with an object of $\Spaces$, rather than a simplicial object, so we adopt the
notational convention that $B(X, G, Y) = |B(X, G, Y)_\bullet|$.

\begin{prop} \label{p:naturalBar} The constructions $B(X,G,Y)_\bullet$, $ B(X,G,Y)$ are natural in all three variables, in the sense that,
  if $G \to G'$ is a homomorphism of group-objects, and if $X$, $Y$ are right- and left-$G$--spaces, and $X'$ and $Y'$ are right- and
  left-$G'$--spaces, such that there are maps $X \to X'$ and $Y\to Y'$ of right- and left-$G'$--spaces, then there is a map $B(X,G,Y)_\bullet
  \to B(X',G',Y')_\bullet$, and similarly for the realization.
\end{prop}

\begin{prop} \label{p:weakBar}
  Suppose that the model structure $M$ has the property that if $A \to A'$ and $B \to B'$ are weak equivalences, then $A \times B \to A'
  \times B'$ is a weak equivalence. In the notation of the previous proposition, if $G\to G'$, $X \to X'$ and $Y \to Y'$ are weak equivalences, then
  $B(X, G, Y) \to B(X', G', Y')$ is a weak equivalence.
\end{prop}
\begin{proof}
  This is straightforward levelwise, and the passage to realizations is effected by proposition \ref{p:hirsch185}.
\end{proof}

We observe that the condition of the previous proposition on the model structure $M$ is met by $\tau$--local model structures, or by the
$\Aone$--local model structure of \cite{MV}.

We remark also that for any $U \in \cat{C}$, and $X,G,Y$ as before, one has the identity
\begin{equation*} 
  B(X,G,Y)(U) = B(X(U),G(U),Y(U))
\end{equation*}
Given an object $Y$ on which $G$ acts on the left, we define a \dfnnd{Borel construction} on $Y$ by
$B (\pt, G, Y) = B( G, Y)$. Our construction is functorial in both $Y$ and $G$.

\begin{prop} \label{p:resent}
  Suppose $G$ is a group object in $\Spaces$ acting on the left on $Y \in \Spaces$. There is a spectral sequence of algebras
  \begin{equation*}
    \Eoh_1^{p,q} = N_p(L^{q,*}( B(G,Y)_\bullet)) 
  \end{equation*}
  which is natural in both $G$ and $Y$, in that a map $(G,Y) \to (G',Y')$ induces a map of spectral
  sequences.

Suppose  that $L^{*,*}$ is bounded below on the set $\{B( G, Y)_s\}_{s=0}^\infty$, then the spectral sequence in question converges strongly:
 \begin{equation*}
    \Eoh_1^{p,q} = N_p(L^{q,*}( B(G,Y)_\bullet)) \Rightarrow L^{p+q,*}(B(G,Y)).
  \end{equation*}

\end{prop}
\begin{proof}
  We apply proposition \ref{p:BasicSS} and corollary \ref{c:BasicSS} to the case of $B(G,Y)$.
\end{proof}

Of course, everything goes through equally well if the $G$--action is on the left, and we obtain the
construction $B(X, G, \pt)$. We understand all subsequent results in this section as asserting also the
equivalent result for such an action \textit{mutatis mutandis\/}.

\subsection{Objects having finitely-generated projective cohomology}
Recall that we write $\LL$ for the cohomology $L^{*,*}(\pt)$.

We encapsulate all the good behavior we typically demand of objects in an omnibus definition:
\begin{dfn}
If an object $Y$ is well-behaved in the sense that $L^{*,*}(Y)$ is finitely-generated projective
$\LL$--module, and if a K\"unneth isomorphism obtains
\begin{equation*}
  L^{*,*}(Y) \tensor_\LL L^{*,*}(X) \isom L^{*,*}(Y\times X)
\end{equation*}
for all objects $X$, then we say $Y$ is \textit{$\LL$--projective}.
\end{dfn}

The K\"unneth spectral sequence
\[\Tor^\LL(L^{*,*}(X), L^{*,*}(Y)) \Longrightarrow L^{*,*}(X \times Y) \]
need not exist in general, but if it does, and if $L^{*,*}(Y)$ is projective over $\LL$, then it is degenerate and the K\"unneth isomorphism
holds.

Especially valuable to us are group-objects $G$ that are $\LL$--projective. In the case of motivic cohomology defined on the $\A^1$--model
structure on $s \Sh_\Nis(\cat{Sm}/k)$---- the setting that eventually will occupy our full attention---examples of such group objects
include finite groups, $\GL_n$ (for which see \cite{DI:MCS} and theorem \ref{t:Pushin}), $\Sl(n)$ and finite products of these groups.

For any $\LL$--module, $N$, we use the notation $\d N$ for $\Hom_{\LL}(N, \LL)$. For a projective $\LL$--module of finite rank, one has $\d
{\d N} = N$. We remark that for two modules, $N_1,\, N_2$, there is a natural map $\d N_1 \tensor_{\LL} \d N_2 \to (N_1 \tensor_{\LL}
N_2)\d{\;}$, which is an isomorphism when both modules are finitely generated and free.

We fix a $\LL$--projective group object $G$, and write $S = L^{*,*}(G)$. The module $\d S$ is in fact a
ring, due to the Hopf-algebra structure on $S$. If $G$ acts on an object $Y$ of $\Spaces$ on the left, then
the action map $G \times Y \to Y$, along with compatibility diagrams, imbues $L^{*,*}(Y)$ with an
$S$--comodule structure. Alternatively, the dual of $L^{*,*}(Y)$ is a module over $\d S$.

\begin{thm} \label{p:MRSSS}
Let $G$ be an $\LL$--projective group object, and let $Y \in \Spaces$ be a simplicial sheaf on which $G$
acts on the left. Suppose $N = L^{*,*}(Y)$ is a graded free $\LL$--module and write $S$ for the Hopf
algebra $L^{*,*}(G)$. There is a spectral sequence 
  \begin{equation*}
    \Eoh_2^{p,q} = \Ext^{p,q,*}_{\d S}(\d N, \LL) 
  \end{equation*}
  which is functorial in $Y$, $G$ and $L^{*,*}$.

  Moreover, the product structure on this spectral sequence is given on the $\Eoh_2$--page
  by the natural product structure on $\Ext^*_{\d S}(\d N, \LL)$, for which see proposition~\ref{p:extprod}.

  If the cohomology theory $L^{*,*}$ is bounded below on the set $\{B( G,Y)_s\}_{s=0}^\infty$, then the spectral sequence in question is strongly convergent
 \begin{equation*}
    \Eoh_2^{p,q} = \Ext^{p,q,*}_{\d S}(\d N, \LL) \Longrightarrow L^{p+q,*}(B( G,Y)).
  \end{equation*}
\end{thm}
\begin{proof}
  Since $L^{*,*}(G)$ is a free $\LL$--module, and $G$ is $\LL$--projective, a K\"unneth isomorphism holds
  \begin{equation*}
    L^{*,*}(B(G,Y)_p) = L^{*,*}(G)^{\tensor p} \tensor_{\LL} L^{*,*}(Y) =  S^{\tensor p} \tensor_{\LL} N.
  \end{equation*}

  
  We now consider the bar complex of $\d N$ as a $\d S$--module, relative to $\LL$.  It is a simplicial $S$--module the $p$--simplices of which
  are
\begin{equation*} \label{eq:BarRes}
	\overbrace{\d S \tensor_{\LL} \d S \tensor_{\LL} \dots \tensor_{\LL} \d
S}^{p+1} \tensor_{\LL} \d N
\end{equation*}
see \cite[chapter
8]{WEIBEL}. Application of $\Hom_{\d S}( \cdot, \LL)$ to this complex yields a cosimplicial $\LL$--module whose $p$--simplices are 
\begin{equation*}
  \Hom_{\d S}( \d S^{\tensor p+1} \tensor_\LL \d N, \LL) \isom S^{\tensor p} \tensor_\LL N,
\end{equation*}
where the natural isomorphism indicated is deduced by elementary ring theory. One can verify by element-level calculations that
the structure maps in this cosimplicial $\LL$--module are precisely those of the cosimplicial $\LL$--module $L^{*,*}(B( G, Y)_
\bullet)$. Since the bar complex is in this case a free resolution of $\d N$ as an $\d S$--module, it follows that the $\Eoh_2$--page of the
spectral sequence is precisely $\Ext^{p,q}_{\d S}( \d N, \LL)$ as promised. \end{proof}




For the most part now we devote ourselves to computing this spectral sequence or its abutment. Our first result towards this end is the
following, which will allow us to compute $L^{*,*}(B( G,Y))$ by decomposition of $Y$.

\begin{prop} \label{p:GColim} Let $G$ be a group object in $\Spaces$. Let $I$ be a small category and let $F\co I \to \Spaces$ be a diagram in
  which all objects $F(i)$ are equipped with a $G$--action and such that the morphisms $F(i \to j)$ are $G$--equivariant. Then there is a weak equivalence $\hocolim |B( G, F(I))_\bullet| \weq |B( G, \hocolim F(I))_\bullet|$.
\end{prop}

The proof relies on commuting homotopy colimits.


\subsection{The Case of a Free Action}\label{s:pseudobundle}

A set-theoretically free action of a group $G$ on a set $X$ is an action for which the stabilizer of every element $x \in X$ is the trivial
subgroup $\{e \} \subset G$. This may be rephrased as the condition that the map $G \times X \to X \times X$ given by $(g,x) \mapsto (gx,
x)$ is an injection of sets. This motivates the following definition:
\begin{dfn}
  If $G$ is a group object in $\Spaces$, acting on an object $Y$ of $\Spaces$, such that the map $G \times Y \to Y \times Y$ is injective,
  then we say the action is \dfnnd{free}.
\end{dfn}

When the action of $G$ on $Y$ is free, one might hope that the simplicial Borel construction and the homotopy type of the quotient
agree. One must be careful not confuse the quotient in $\Spaces$ with other notions of quotient that may exist internally to $\cat{C}$.


We are trying to imitate the following fact, true in the context of simplicial sets:
\begin{prop}
  Let $X$ be a simplicial set, and $G$ a simplicial group, acting freely on $X$. Then there is a map
  $B( G, X) \to X/G$ which is a weak equivalence. This map is natural, in that $(G,X) \to
  (G',X')$ induces a diagram
  \begin{equation*}
    \xymatrix{ B( G, X) \ar[r] \ar[d] & X/G \ar[d] \\ B( G',X') \ar[r] & X'/G'.}
  \end{equation*}
\end{prop}
See \cite[Chapter 8]{MAYCLASS} for the proof.

We specify that if $G$ is a group object in $\Spaces$ and if $X$ is an object in $\Spaces$ on which $G$ acts, then the notation $X/G$ is to
mean the `orbit sheaf' in $\Spaces$. This is the quotient sheaf associated to the presheaf
\begin{equation*}
U \mapsto \operatorname{coeq}\Big( G(U) \times X(U) \rightrightarrows X(U)\Big),
\end{equation*}
\textit{viz.}~the presheaf of ordinary group-quotients. Even when $X$, $G$ are represented by objects of $\cat{C}$, for instance if they are
sheaves represented by schemes, we shall never write $X/G$ for any other quotient
than this sheaf-theoretic quotient.

\begin{dfn}
  Let $Y \to X$ be a surjective map in $\Spaces$ and let $G$ be a group object in $\Spaces$ such that there is an action of $G$ on
  $Y$, denoted $\alpha\co G \times Y \to Y$, and that $G$ acts trivially on the object $X$, and that the map $Y \to X$ is $G$
  equivariant. This is equivalent to saying the outer square in the diagram below is commutative
  \begin{equation*}
    \xymatrix{ G \times Y \ar@/_18px/_{\pi_2}[dr] \ar@/^18px/^{\alpha}[rr] \ar@{.>}^{\Phi}[r] & Y  \times_X Y \ar[r] \ar[d]&  Y \ar[d] \\ &  Y  \ar[r] & X}
  \end{equation*}
  Under these assumptions there exists a natural map $\Phi\co G \times Y \to Y \times_X Y$. If $\Phi$ is
  an isomorphism, then we say that $Y \to X$ is a \dfnd{principal $G$--bundle}. We abbreviate the data of a $G$--bundle to $Y \to X$, where
  the $G$ action on $Y$ is understood.
\end{dfn}

A map of $G$--bundles $Y \to X$ to $Y' \to X'$ is a pair of $G$--equivariant maps $Y \to Y'$ and $X \to X'$ making the obvious square
commute. Note that the $G$--equivariance of $X \to X'$ is always trivially satisfied.

There are two maps $G \times Y \to Y$, both $\alpha$, the action map, and $\pi_2$, projection on the second factor. In tandem, they yield a
map $\alpha \times \pi_2 \co G \times Y \to Y \times Y$, which is commonly denoted $\Psi$. When we refer to a map $G \times Y \to Y \times Y$
without expressly naming one, it is to be understood that $\Psi$ is intended.

\begin{prop} \label{p:pseudofree}
  If $Y \to X$ is a principal $G$--bundle, then the action of $G$ on $Y$ is free.
\end{prop}
The proof is immediate from the case of sets, since monomorphisms of sheaves may be detected objectwise.


\begin{prop} \label{p:freeBar}
  Let $X$ be an object in $\Spaces$, let $G$ be a group object in $\Spaces$, and suppose the action of $G$ on $X$ is free. Then there is an injective
  weak equivalence $|B( G, X)| \weq X/G$ in $\Spaces$. This weak equivalence is natural in
  both $G$ and $X$.
\end{prop}
\begin{proof}
  For any $U \in \cat{C}$, we have $B(G,X)(U) = B( G(U), X(U))$, and so it follows that
  \begin{equation*}
    B( G, X)(U) \weq X(U)/G(U),
  \end{equation*}
  and $X(U)/G(U) = (X/G)^{\text{pre}}(U)$, the latter being the
  presheaf quotient. It follows that $|B( G,X)| \weq (X/G)^{\text{pre}}$ in the model category
  of presheaves. By \cite{JARSIMPPRE}, we know that $(X/G)^{\pre} \weq X/G$, the latter being the
  quotient sheaf for the $\tau$--topology.
\end{proof}

One source of bundles is the following: If $G \in \Spaces$ is a group object, and $G$ acts on $X \in \Spaces$, and if $X$ has a point
$x_0\co \pt \to X$, then there is a map $G \times \pt \to G \times X \to X$. In the composition, we denote this by $f\co G \to X$. Let $H$ be
the pull-back as given below \begin{equation*} \xymatrix{ H \ar[r] \ar[d] & G \ar^f[d] \\ \pt \ar^{x_0}[r] & X}
\end{equation*}
then $H$ is called the \dfnd{stabilizer} of $x_0$ in $G$. It is easily seen that $H$ is a sub-group-object of $G$. As a result, there is a group action (on the right) of $H$ on $G$ given by multiplication $G \times
H \to G$.

\begin{prop} \label{p:subgrouppseudo}
  With notation as above, if the map $G \to X$ is surjective then $G \to X$ is a principal $H$--bundle.
\end{prop}



The proof proceeds by checking the conditions on sections; there they are elementary. The surjectivity of $G \to X$ is a transitivity
condition in the category of sheaves.

\subsection{Quotients of Representables}

Recall that there is a Yoneda embedding $h\co\cat{C} \to s\Pre(\cat{C})$. We say that a simplicial sheaf is \dfnnd{representable} if it is in
the essential image of this embedding. We do not demand that all objects of $\cat{C}$ represent sheaves. We shall generally denote an object
of $\cat{C}$ and the presheaf or sheaf it represents by the same letter.

The following lemma is elementary, and may be proved by considering points.

\begin{lemma} \label{l:technicalCoeq}
  Let $\pi\co A \to B$ be a map in $\Spaces$, then there are two maps $A \times_B A \to A$, being the
  projection on the first and second factor respectively. Suppose $A \to B$ is a simplicial sheaf epimorphism,
  then the  following
  \begin{equation*}
    \xymatrix{ \cqld{ A \times_B A }{ A}{ B}{}{}{} }
  \end{equation*}
  is a coequalizer diagram of simplicial sheaves.
\end{lemma}

\begin{prop} \label{p:bundle}
  Let $G$ be a group-object in $\Spaces$, which acts on the object $Y$, and trivially on the 
  object $X$. Let $\pi\co Y \to X$ be a $G$--equivariant map, making $Y$ a principal $G$--bundle over $X$.
  Then there is an isomorphism $Y/G \isom  X$, which is natural in $G, Y $ and $X$.
\end{prop}
\begin{proof}
Using lemma
\ref{l:technicalCoeq}, we see that in the diagram
\begin{equation*}
  \xymatrix{ \cqld{ Y \times_X Y}{Y \ar@{=}[d]}{ X }{}{}{}\\ \cqld{G \times Y \ar^{\Phi}[u]}{ Y}{Y/G \ar^\phi[u]}{}{}{}}
\end{equation*}
both sequences are coequalizer sequences. By assumption $\Phi\co G \times Y \to Y \times_X Y$ is an isomorphism. By categorical uniqueness of
coequalizers, the map $\phi\co Y/G \to X$ is an isomorphism. The asserted naturality follows immediately by considering diagrams of
coequalizer sequences, and is routine.
\end{proof}

\begin{cor} \label{Cor1.12.1} Let $G$ be a group-object in $\Spaces$, which acts on the object $Y$, and trivially on the object $X$. Let
  $\pi\co Y \to X$ be a principal $G$--bundle. Then there is a weak equivalence $B( G, Y) \weq X$, which is natural in $G, Y$ and $X$.
\end{cor}
\begin{proof}
  This follows from the weak equivalence $B( G, Y) \weq Y/G$.
\end{proof}

Suppose that if $X$, $Y$ and $G$ are representable objects in $\Spaces$. If we consider them as objects in $\cat{C}$, then of all the
hypotheses of the two results above, only the statement that $\pi\co Y \to X$ is a surjective map of sheaves cannot be verified in $\cat{C}$
without reference to the topology. To verify that a map $\pi\co Y \to X$ is a surjective map of sheaves, we must find a family of maps $\{
f_j\co U_j \to X \}$ which are covering for the topology $\tau$, and which have the property that there are sections $s_j\co U_j \to Y$
satisfying $\pi \circ s_j = f_j$, in this case we say $\pi$ admits \dfnnd{$\tau$--local sections}.

Combining the above with proposition \ref{p:MRSSS} gives the following.

\begin{theorem} \label{p:mainres} Let $G$ be an $\LL$--projective representable group-object, and abbreviate the cohomology as $S =
  L^{*,*}(G)$. Suppose we have a map of representable objects, $\pi\co Y \to X$, which admits $\tau$--local sections, and which is a principal
  $G$--bundle. Suppose that $G$ and $Y$ belong to a subcategory $\cat{U}$ of $\Spaces$ which is closed under formation of products and such
  that $L^{*,*}$ is $\cat{U}$--bounded-below. Suppose further that the cohomology of $Y$, denoted $N = L^{*,*}(Y)$, is free and finitely
  generated as an $\LL$--module. There exists a strongly convergent spectral sequence of algebras
  \begin{equation*}
    \Eoh_2^{p,q} = \Ext^{p,q}_{\d S}(\d N, \LL) \Longrightarrow L^{p+q, *}(X;R),
  \end{equation*}
  and this spectral sequence is functorial in $(G,Y)$. The differentials take the form $d_r\co \Eoh_r^{p,q} \to \Eoh_r^{p+r, q-r+1}$.
\end{theorem}

The functoriality is exactly the usual functoriality of $\Ext$.

\section{The Equivariant Cohomology of Stiefel Varieties}

The theme of this section is the application the tools of the previous section to the case of actions of $\G_m$ on Stiefel
varieties. The signal results are the almost-complete calculation (by which we mean the identification of the $\Eoh_\infty$--page of a convergent
spectral sequence) of the motivic cohomology of $B( \G_m, \GL_n)$ with a general $\G_m$--action in proposition \ref{p:wasCor4}, and the partial
calculation of the analogue with $\GL_n$ replaced by a more general Stiefel variety in theorem \ref{th:EqCohStiefel}.

In classical topology, one might consider the fibration $X \to X \times_G EG \to BG$ and then employ a Serre spectral sequence to go from
knowledge of $H^*(X)$ and $H^*(BG)$ to $H^*(X \times_G EG)$. This is what we do in spirit, since if we were to take that fiber sequence and
start to extend it to a Puppe sequence, we should arrive at $G \weq \Omega BG \to X \to X \times_G EG$, which is equivalent to the fiber
sequence $G \to EG \times X \to EG \times_G X \weq B( G, X)$. As a consequence of our going this roundabout way, the spectral sequences we
obtain have $\Eoh_2$--pages resembling the $\Eoh_3$--pages of the Serre spectral sequences for which they are substitutes.

\subsection{Generalities}

We calculate with simplicial sheaves in the $\A^1$--model structure on the category $\nSpaces$. The cohomology theory employed is motivic
cohomology, $U \mapsto \Hoh^{*,*}( U ; R)$; for a full treatment of this theory see \cite{MVW}, \cite{MOTEMSP}. For any commutative ring $R$,
there exists a bigraded cohomology theory $\Hoh^{*,*}( \cdot ; R)$, the definition of which is functorial in $R$.

Throughout $k$ will denote a field, on which further restrictions shall be placed later. By $\cat{Sm}_k$ we mean, as is customary,
\cite{MV}, the category of smooth, separated, finite type $k$--schemes; the objects of this category will be referred to simply as
\dfnnd{smooth schemes} in the sequel. The topology is the Nisnevich topology, for which see \cite{NIS} or \cite{MV}.

We enumerate some of the properties of $\Hoh^{*,*}$. In particular, one has the following:
\begin{prop}
  Let $U$ be a smooth scheme, then $\Hoh^{p,q}(U ; R) = \CH^q(U, 2q-p)_R$, where the latter is the $R$--valued higher Chow groups of \cite{BLOCH86}.
\end{prop}

In order to use our convergence results, we will need to know some boundedness result for motivic cohomology. The
\textit{Beilinson-Soul\'e vanishing conjecture\/} is that for a smooth scheme $X$, one has $\Hoh^{p,q}(X; \Z) = 0$ when $p<
0$. This is known only in certain contexts at present.

The following result is folklore. It is proved by using the Bloch-Kato conjecture, which is now proved, to reduce the question for all
fields to the problem with $\Q$--coefficients. With $\Q$--coefficients, the motivic spectral sequence converging to $K$--theory is degenerate,
and for finite fields or number fields the higher $K$--theory with $\Q$--coefficients is known.

\begin{prop}
  The Beilinson-Soul\'e vanishing conjecture holds for the motivic cohomology $\Hoh^{p,q}(\Spec F; \Z)$ when $F$ is a finite field or a number field.
\end{prop}

For any commutative ring $R$, the ring $\Hoh^{*,*}(\Spec k; R)$ will be denoted $\M_R$ or $\M$ where $R$ is understood. We shall write
$\M^{i,j}$ to denote $\Hoh^{i,j}(\Spec k; R)$. Observe that $\M^{0,0} = R$. Note that for any $Y$, the ring $ N = \Hoh^{*,*}(Y; R)$ is an
$\M_R$--module. We shall write $\d N$ to denote the dual, $\d N = \Hom_{\M_R} ( N, \M_R)$. In addition to the Beilinson-Soul\'e vanishing
conjecture, the following vanishing obtains:
\begin{prop}
  For any $X \in \Sm_k$, one has $\Hoh^{p,q}(X ; R)  = 0$ if $q<0$, or if $2q-p < 0$. In the case where $X$ is equidimensional of dimension
  $d$, one has $\Hoh^{p,q}(X;R) = 0$ when $p- d -q< 0$.
\end{prop}

In the following theorem, `stably cellular' is taken in the sense of \cite{DI:MCS}.

\begin{theorem} \label{t:mainMCT}
  Let $G$ be a smooth, stably-cellular group scheme over $k$ such that $S = \Hoh^{*,*}(G; R)$ is a finitely generated graded free $\M_R$--module,
  generated by elements in nonnegative bidegree. Suppose $G$ acts on a smooth scheme $Y$ such that $ N =\Hoh^{*,*}(Y; R)$ is also a finitely
  generated free $\M_r$--module, again generated by elements in nonnegative bidegree. Then there is a strongly convergent spectral sequence of algebras:
  \begin{equation*}
    \Eoh_2^{p,q} = \Ext_{\d S}( \d N, \M_R) \Longrightarrow \Hoh^{*,*}( B( G, Y) ;R),
  \end{equation*}
  which is functorial in $G$, $Y$ and $R$. The differentials act as $d_r\co \Eoh_r^{p,q} \to \Eoh_r^{p+r, q-r+1}$.
\end{theorem}
\begin{proof}
  This is a special case of theorem \ref{p:MRSSS}.
\end{proof}

\begin{cor} \label{c:MCT}
  We continue the hypotheses of the theorem, and add the following: if the action of $G$ on $Y$ is free and if $X$ is a scheme representing
  the Nisnevich quotient $Y/G$ then the spectral sequence above converges to the motivic cohomology, $\Hoh^{*,*}(X; R)$:
  \begin{equation*}
    \Eoh_2^{p,q} = \Ext_{\d S}( \d N, \M_R) \Longrightarrow \Hoh^{*,*}( X ;R).
  \end{equation*}
\end{cor}
\begin{proof}
  A free action is one for which the map $G \times Y \to Y \times Y$ is a monomorphism. Since the Yoneda embedding preserves monomorphisms,
  an action which is free in the scheme-theoretic sense is free in the sheaf-theoretic sense. The corollary follows from the theorem in
  conjunction with proposition \ref{p:freeBar}.
\end{proof}

\begin{theorem} \label{c:MCHS}
  Let $G$ be a smooth group-scheme, and $X$ a smooth scheme on which $G$ acts. Suppose $x_0 \to X$ is a $k$--point of $X$ and $H$ is the
  stabilizer of $x_0$. Write $S = \Hoh^{*,*}(H; R)$, and $N = \Hoh^{*,*}(G; R)$. Suppose both $S$ and $N$ are finitely generated free graded
  $\M_R$--modules, generated by classes of nonnegative bidegree. Denote by $f$ the obvious map, $f\co G \times x_0 \to G \times X \to X$. If this map is
  Nisnevich-locally split, in that there is a Nisnevich cover $c\co U \to X$ and a splitting map $s\co U \to f^{-1}(U)$ satisfying $c^{-1}(f)
  \circ s = \id_U$, then there is a strongly convergent spectral sequence of algebras
  \begin{equation*}
    \Eoh_2^{p,q} = \Ext_{\d S}(\d N, \M_R) \Longrightarrow \Hoh^{*,*}(X; R).
  \end{equation*}
\begin{proof}
  Since the Yoneda embedding commutes with the formation of limits, $ G \to X$ is a principal $H$--bundle in
  $s\Sh_\Nis(\Sm/k)$. The result follows from proposition \ref{p:subgrouppseudo} and theorem \ref{p:mainres}.
\end{proof}
\end{theorem}

The cohomology $\Hoh^{*,*}(B( G, Y) ; R)$ is a variant of the Borel-equivariant cohomology with respect to $G$. There is another, more
geometric, definition of equivariant higher Chow groups, defined in \cite{EG}, which is denoted $\CH^*_G(X, *)_R$. In general, for a group-scheme for which
$\Hoh^1_\et(\cdot, G) \not \isom \Hoh^1_\Nis(\cdot , G)$, the two definitions are different; we can at least offer the
following comparison result which applies to the cases considered in the sequel. 

Recall that a group $G$ is \dfnnd{special} if every principal \'etale $G$--bundle is locally trivial in the Zariski (and {\it a fortiori} in
the Nisnevich) topology. The group-schemes $\GL_n$ and $\Sl_n$ are special, as are products of special group-schemes \cite{SEM-CHEV-TORS}.

The following is a version of a result from \cite[Appendix A]{ARASON}. The vanishing conjecture is
included in order to ensure convergence of the spectral sequences that appear.

\begin{prop} \label{p:smoothGroup}
  Let $G$ be a special group-scheme. Let $X$ be a smooth scheme on which $G$ acts, and suppose the Beilinson-Soul\'e
  vanishing conjecture is known to hold for all spaces of the form $G^{\times n}$ and $G^{\times n}
  \times X$, including the special cases where $n=0$. Then $\Hoh^{p,q}( B( G, X) ; R) \isom \CH^q_G(X, 2q-p)_R$.
\end{prop}
\begin{proof}
  One calculates $\CH^i_G(X, 2i-n)_R$ by finding a representation $V\isom \A^N$ of $G$ such that the locus $Z \subset V$ where the action of
  $G$ on $V$ is not free is of very high codimension (codimension in excess of $2i+2$ will be sufficient for our purposes), and such that
  $V \to V/G$ is Zariski-locally trivial.  Write $Q = (X \times ( V \sm Z))/_\et G$, where the quotient $/_\et$ is the algebraic-space
  quotient. We know from \cite[Proposition 23]{EG} that $Q$ is in fact a scheme. The quotient map $X \times (V \sm Z) \to (X
  \times (V\sm Z))/_\et G$ is surjective as a map of \'etale sheaves, and it follows that $X \times (V \sm Z) \to (X \times (V \sm Z)/_\et
  G$ is an \'etale-locally trivial principal $G$--bundle. Since $G$ is special, it is in fact a Nisnevich-locally trivial principal $G$--bundle. In
  particular, $Q \weq B( G, X \times (V \sm Z)) $ by corollary \ref{Cor1.12.1}.

  One defines $\CH^i_G( X, 2i-n)_R =\CH^i( Q , 2i-n)_R$.

  For the groups $G$ satisfying the hypotheses above, one may ensure further that $Z$ is a union of linear subspaces of $V$, in particular
  $V \sm Z$ is cellular, and has cohomology which is finitely generated and free over $\M$. Therefore $X \times (V \sm Z)$ satisfies the Beilinson-Soul\'e
  vanishing conjecture if $X$ does.

  Proposition \ref{p:resent} establishes convergent spectral sequences 
  \[N_p( \Hoh^{q, *}( B( G, X);R  )_\bullet) \Rightarrow \Hoh^{p+q}( B( G, X))\] and \[N_p(\Hoh^{q,*}(B( G, X \times (V \sm Z)); R) ) \Rightarrow \Hoh^{p+q}(Q; R).\]
  On the other hand, $\Hoh^{q,*}(B( G ,X)_s ; R) = \Hoh^{q,*}(B( G, X \times (V \sm Z))_s; R)$ for $0 \le s \le 2i+1$ by hypothesis on $V$,
  so that the two convergent spectral sequences agree on every page in the region $p+q \le i$. Since both converge strongly, the result
  follows.
\end{proof}

Examples where the two theories differ are abound once we drop the hypothesis that the group scheme be special. Suppose
that we consider $G=\mu_2\isom \Z/2$ acting on $\G_m$ by $x \mapsto -x$, all over $\Spec \C$ for concreteness. Then the \'etale-sheaf quotient
$\G_m/G$, which is also the geometric quotient, is $\G_m$. It follows that $\CH^*_G(\G_m, *) \isom \CH^*(\G_m,*)$, although the isomorphism
is not via the projection map $\G_m \to \G_m/G$. If we consider our Borel equivariant motivic cohomology, however, we obtain a spectral
sequence from theorem \ref{t:mainMCT}:
\[ \Eoh_2^{*,*} = \Ext_{\M[G]}( \M \oplus \M\sigma, \M) \Longrightarrow \Hoh^{*,*}(B(G,\G_m); \Z). \] Here $\sigma$ is the dual of the
generator of $\tilde{\Hoh}^{*,*}(\G_m; \Z)$, one has $|\sigma|=(1,1)$, and one verifies (see proposition \ref{p:signchange}) that $G$ acts
trivially on $\M \oplus \M\sigma$. The ring $\M[G]$ may be presented as $\M \oplus \M g$, with $g^2=1$, and $|g|=(0,0)$. The trigraded
$\Eoh_2$--page may be described as $(\M \oplus \M \d \sigma) \tensor_\Z \Hoh_{\textrm{group}}^{*}(G, \Z)$, where the
group-cohomology lies in tridegrees $(*,0,0)$ and $\d \sigma$ lies in tridegree $(0,1,1)$. In particular, $\Eoh_2^{2i,0,0} = \Z/2$ whenever
$i>0$, but there are no elements of weight $0$ in $\Eoh_2^{*,p}$ where $p>0$. The classes in $\Eoh_2^{2i, 0,0}$ for $i>0$ cannot therefore be the
source or target of any differential, and therefore persist to the $\Eoh_\infty$--page. Since the spectral sequence converges, $\Hoh^{2i,
  0}(B(G, \G_m); \Z) = \Z/2$ for $i \ge 0$, and $B(G, \G_m)$ does not have the motivic cohomology of a smooth scheme.

\subsection{The Ordinary Motivic Cohomology of Stiefel Varieties}

Let $i$ and $n$ be positive integers satisfying $i \le n$. We define a Stiefel variety $V_i(\A^n)$ to be the variety of $n \times i$--matrices having rank
$i$. There are two special cases: $V_1(\A^n) = \A^n \sm {0}$ and $V_n (\A^n) = \GL_n$. The Stiefel variety $V_i(\A^n)$ is an open subscheme of
$\A^{in}$.

The following result is due to \cite{PUSHIN}. We write $\{a\} \in \Hoh^{1,1}(\Spec k; \Z)$ for the element corresponding to $a \in k^*$,
and we also write $\{a\}$ for the image of this class in $\Hoh^{1,1}(X; R)$ under the natural map induced by $X \to \Spec k$ and $\Z \to R$.

\begin{theorem}\label{t:Pushin}
  Over any field, the motivic cohomology $\Hoh^{*,*}(\GL_n; R)$ is the almost-exterior algebra generated over $\M$ by classes $ \rho_1
  ,\dots, \rho_n$ in bidegrees $|\rho_i| = (2i-1,i)$ and subject to the relation
  \begin{equation*}
    \rho_i^2 = \begin{cases} 0 \quad \text{ if $2i - 1 > n$} \\ \{-1\} \rho_{2i-1} \quad \text{ otherwise} \end{cases}
  \end{equation*}
  and to the usual constraints of graded-commutativity in the first grading and commutativity in the second.
\end{theorem}

One may deduce the following corollaries, \cite{WILLIAMS1}:
\begin{cor} \label{p:Wincl} Over any field, the motivic cohomology $\Hoh^{*,*}(V_m(\A^n); R)$ is the subalgebra of $\Hoh^{*,*}(\GL_n; R)$
  generated by $ \rho_{n-m+1}, \dots, \rho_n$. The usual projection $\GL_n \to V_m(\A^n)$ induces this inclusion on cohomology.
\end{cor}

\begin{prop} \label{p:GLNsurj}
  The usual inclusion $\GL_n \to \GL_{n+1}$ induces the quotient map \[\Hoh^{*,*} (\GL_{n+1}; R) \to \Hoh^{*,*}(\GL_{n+1}; R) / ( \rho_{n+1}) \isom \Hoh^{*,*}( \GL_n ;R).\]
\end{prop}

\begin{prop}
  The group multiplication $\mu\co \GL_n \times \GL_n \to \GL_n$ induces a comultiplication $\mu^*\co \Hoh^{*,*}(\GL_n; R) \to \Hoh^{*,*}(\GL_n;
  R) \tensor_\M \Hoh^{*,*} (\GL_n; R)$. This is an algebra map, and is fully determined by $\mu^*(\rho_i) = 1 \tensor \rho_i + \rho_i \tensor 1$.
\end{prop}
\begin{proof}
  This follows immediately from the bigrading on $\Hoh^{*,*}(\GL_n; R)$.
\end{proof}

\begin{cor} 
  The group-inversion map $i\co \GL_n \to \GL_n$ induces a map $i^*$ on cohomology that is determined by $i^*(\rho_i) = -\rho_i$.
\end{cor}
\begin{proof}
  Observe that the composite of the diagonal $\Delta\co \GL_n \to \GL_n \times \GL_n$, inversion on the first factor $\iota \times \id$, and
  multiplication $\mu\co \GL_n \times \GL_n \to \GL_n$ gives the $0$--map on cohomology. The result follows.
\end{proof}

\begin{cor} \label{p:gmactongln}
  The action $\alpha\co \G_m \times \GL_n \to \GL_n$ given by multiplication of a row or column by a scalar induces a coaction
  \begin{equation*}
  \alpha^* \co  \Hoh^{*,*}(\GL_n; R) \to \Hoh^{*,*}(\G_m; R) \tensor \Hoh^{*,*}(\GL_n; R)
  \end{equation*}
  which is given by $\alpha^*(\rho_1 ) = \rho_1 \tensor 1 + 1 \tensor \rho_1$, and $\alpha^*(\rho_i) = 1 \tensor \rho_i$.
\end{cor}

\begin{prop} \label{p:signchange} 
  There is a $\mu_2$ action on $\A^n \sm \{0\}$ given by $\vec v \mapsto - \vec v$. The
  induced $\mu_2(k)$--action on $\Hoh^{*,*}(\A^n\sm \{0\}; R)$ is trivial.
\end{prop}
\begin{proof}
   The $\mu_2(k)$--action on the whole ring may be deduced from the $\mu_2(k)$ action
  on $\Hoh^{2n-1,n}(\A^n \sm \{0\}; R) = \CH^n( \A^n \sm \{0\}, 1)_R$, since this group and
  $\CH^0(\A^n\sm\{0\}, 0)_R$ serve to generate the ring. Since the $\mu_2(k)$--action extends to an
  action on $\A^n \weq \pt$ fixing the origin, the result follows from functoriality of the
  localization sequence for higher Chow groups.
\end{proof}

\begin{prop} \label{p:permutInvGln}
  There is a symmetric group, $\Sigma_i$, action on $V_i(\A^n)$, given by permuting the columns of matrices. This action is trivial on cohomology.
\end{prop}
\begin{proof}
  By considering first the projection $\GL_n \to V_i(\A^n)$ and then the inclusion $\GL_n \to \GL_{n+1}$, one may assume $V_i(\A^n)$ is in
  fact $\GL_n$ for $n>3$. A transformation of the form $A \mapsto A \cdot E_{ij}(\lambda)$ where $E_{ij}(\lambda)$ is an elementary matrix
  induces the identity map $\Hoh^{*,*}(\GL_n ; R) \to \Hoh^{*,*}(\GL_n; R)$, since there is an $\A^1$--homotopy of maps from multiplication-by-$E_{ij}(\lambda)$
  to the identity. The transformation that interchanges two columns and changes the sign of one is a composite of such elementary
  transformations. By comparison with proposition \ref{p:signchange} we see that the change of sign also induces the identity on cohomology,
  so that it follows that all transpositions of columns induce the identity on cohomology. Since these serve to generate $\Sigma_i$, the result follows.
\end{proof}

\subsection{Torus Actions on $\GL_n$}

We compute some of examples of the Rothenberg-Steenrod spectral sequence for motivic cohomology, culminating in the case of $B( \G_m, \GL_n)$ for
a $\G_m$--action on $\GL_n$ of the most general type. This is the content of proposition \ref{p:lrsss}. In general, the determination of the
$\Eoh_2$--page of the spectral sequences is simply a calculation of $\Ext$--groups, is not difficult, and relies on a reference to proposition
\ref{p:extCalc} or \ref{p:extCalc2}. The determination of the differentials is more involved, and relies on the nature of the objects being
studied.

Unless otherwise stated, we assume that all varieties are defined over a ground-field $k$ such that the Beilinson-Soul\'e vanishing conjecture is known to hold for
$\Hoh^{*,*}(\Spec k;\Z)$. We shall also assume that our ring of coefficients, $R$, is one in which $2$ is invertible or else that $R$ is a ring of
characteristic $2$ and $k$ contains a square-root of $--1$. In practice, we shall use $R = \Z[\frac{1}{2}]$, $R = \Z/p$ where $p$ is
an odd prime, $R=\Q$ and $R= \Z/2$, but the last only in the case where $--1$ is a square in $k$. Under either hypothesis, the element $\{-1\} \in \M_R^{1,1}$
vanishes, and the cohomology $\Hoh^{*,*}( \GL_n; R)$ is an exterior algebra.

This section makes extensive use of the trigraded nature of the motivic spectral sequences, and it is therefore convenient to have a
notational convention for that grading. A homogeneous element $\alpha$ in the $j$--th page of a spectral sequence will be said to have tridegree
$|\alpha| = (p,q,r)$ if it lies in homological degree $p$, motivic degree $q$ and weight $r$. This element corresponds to one
that would classically be understood to lie in bidegree $(p,q)$, that is to say $p$ `across' and $q$ `up'. The differential $d_j$ invariably
will take $\alpha$ in of tridegree $(p,q,r)$ to $d_j \alpha$ in tridegree $(p+j, q-j+1, r)$. We define the \textit{total Chow height\/} of
$\alpha$ to be $\tchh{\alpha} = 2r - p-q$, and we note that
\begin{equation} \label{e:tchh}
  \tchh{d_j \alpha} = \tchh{\alpha} - 1 .
\end{equation}
Since total Chow height is linear in each grading, we also have $\tchh{\alpha \beta} = \tchh{\alpha}
+ \tchh{\beta}$. In general, equation \eqref{e:tchh} allows us to discount a great many potential
differentials in the motivic spectral sequence, which is an advantage over the classical case,
where we do not have the crutch of the weight filtration.

The first, and easiest, of the spectral sequences is the following:
\begin{prop} \label{p:AnPn}
  Let $\G_m$ act on $\A^n \sm \{0\}$ via the diagonal action, that is to say the map given by $r \cdot  (a_1, \dots, a_n) = (ra_1, \dots, ra_n)$. 

  Suppose $n > 1$, then the $\Eoh_2$--page of the associated spectral sequence in motivic cohomology is the $\M$--algebra $\M[\rho_n, \theta]/
  (\rho_n^2)$, with $|\rho_n| = (0,2n-1,n)$ and $|\theta| = (1,1,1)$. There is a single nonvanishing differential of note, on the $n$--th
  page, satisfying $d_n \rho_n = \theta^n$. All other differentials are determined by this one.

  Suppose $n=1$, then $\G_m \isom \A^1 \sm \{0\}$ and the spectral sequence is trivial.
\end{prop}
\begin{proof}
  We consider only the case $n>1$, the case $n=1$ being trivial.

  The group action of $\G_m$ on $\A^n \sm \{0\}$ gives rise to the principal $\G_m$--bundle $\A^n \sm \{0\} \to \P^{n-1}$, which is
  Zariski-locally trivial; it follows from proposition \ref{p:bundle} that $|B( \G_m, \A^n \sm \{0\})| \weq \P^{n-1}$.

  The calculation of the $\Eoh_2$--page of the spectral sequence is straightforward, by reference to corollary \ref{p:gmactongln} and then
  proposition \ref{p:extCalc}. The $\Eoh_2$--page is therefore $\M[\rho_n, \theta]/ (\rho_n^2)$, with $|\rho_n| = (0,2n-1,n)$ and
  $|\theta| = (1,1,1)$, and we need only determine the differentials supported by $\rho_n$ and $\theta$.
  \begin{figure}[htb]
    \centering
    \begin{equation*}
      \xymatrix@R=2px@C=10px{ 
        \ast  & \ast &\ast  & & \ast \\ 
        \M^{0,0} \rho_n \ar^{d_n}[drrrr] & \ast & & & \ast \\
        \ast & \ast & \ast & & \M^{0,0}\theta^n \\
        \vdots & \vdots & \vdots & \iddots  & 0\\
        \ast & \ast &  \M^{0,0} \theta^2  &  & 0  \\
        \ast & \M^{0,0} \theta  & 0 &  & 0 \\
        \M^{0,0} & 0 & 0 &  & 0 }
    \end{equation*}
    \caption{ \small The $\Eoh_2$--page of the spectral sequence converging to $\Hoh^{*,*}(\P^{n-1}; R)$. Only the elements on the $(i,i)$ diagonal
      lie in Chow height $0$. The lower triangle of zeroes is implied by Beilinson-Soul\'e vanishing.}
    \label{figPn-1}
  \end{figure}

  For dimensional reasons, $\theta$ cannot support any nonzero differentials. The total Chow height of $\rho_n$ is $1$, and so if it is to
  support a differential, the image must be of total Chow height $0$. We have $\tchh{\theta} = 0$, and so $\tchh{\theta^i} = 0$. For any
  element of nonzero degree, $\mu$, in $\M$, one has $\tchh{\mu \theta^i} > 0$. If $\rho_n$ is to support a differential, it must take the
  form $d_j \rho_n = \ell\theta^i$, where $\ell \in \M^{0,0}$. Considering degrees we should have $(j, 2n-j, n) = (i,i,i)$, so
  $i=j=n$. Since $\rho_n$ can support no other differential, the spectral sequence collapses by the $n+1$--st page at the latest. The
  sequence converges to the motivic cohomology of $\P^{n-1}$, for which the corresponding group $\Hoh^{2n,n}(\P^{n-1};R)$ is $0$, it follows
  that $\ell$ is a unit. Without loss of generality, we can choose generators for the cohomology of $\G_m$ and of $\A^{n-1}\sm\{0\}$ so that
  this unit is in fact $1$.
\end{proof}

We write $A^{(n)}$ for the $n$--graded part of a graded object $A$. In an abuse of notation, if $I \subset A$ is an ideal, we will write $A^{(n)}/I$
for the quotient group $A^{(n)}/(I \cap A^{(n)})$.

\begin{prop} \label{p:wasCor2}
  Suppose $n\ge 1$. Let $T_n = (\G_m)^n$ act on $\A^n \sm \{0\}$ via the action $(\lambda_1, \dots,
  \lambda_n) \cdot (\mu_1, \dots, \mu_n) = (\lambda _1 \mu_1, \dots,  \lambda_n \mu_n)$. Then the
  $\Eoh_2$--page of the associated spectral sequence for motivic cohomology is
  the $\M$--algebra $\M[\rho_n, \theta_1, \dots, \theta_n]/(\rho_n^2)$. There is a single
  nonvanishing differential of note, on the $n$--th page, satisfying $d_n \rho_n = \prod_{i=1}^n
  \theta_i$. All other differentials are determined by this one.
\end{prop}
\begin{proof}
  The proof proceeds by induction on $n$. The case of $n=1$ has been handled above.
  
  The determination of the $\Eoh_2$--page is again straightforward, following from corollary \ref{p:gmactongln} and proposition \ref{p:extCalc}. For dimensional
  reasons the elements $\theta_i$ cannot support nonzero differentials. The calculation reduces to the question of differentials supported
  by $\rho_n$.

  There is a map of group-schemes $\Delta\co \G_m \to T_n$, given by the diagonal, and a commutative diagram of group actions
 \begin{equation*}
 {}_I \Eoh_* \left\{ \raisebox{2em}{\mbox{ \xymatrix{ \G_m \times \A^n \sm \{0\} \ar^{\Delta \times \id}[r] \ar[d] & T_n \times \A^n \sm
   \{0\} \ar[d] \\ \A^n \sm \{0\} \ar@{=}[r] & \A^n \sm \{0\}}}}\right\} {}_{II} \Eoh_*
 \end{equation*}
 from which it follows that there is a map of spectral sequences, which we denote by $\Delta^*$. We refer to the spectral sequence for the
 $\G_m$--action as the first spectral sequence, that for the $T_n$--action as the second, and write ${}_I\Eoh^{*,*,*}_*$,
 ${}_{II}\Eoh^{*,*,*}_*$ to distinguish them. The map $\Delta^*$ goes from ${}_{II}\Eoh_*$ to ${}_I\Eoh_*$. We have $\Delta^*( \rho_n) =
  \rho_n$ and $\Delta^*(\theta_i) = \theta$, the second conclusion following from the effect of the map $\Delta^*\co \Hoh^{*,*}(T_n;\Z) \to
 \Hoh^{*,*}(\G_m;\Z)$ on cohomology along with some rudimentary homological algebra.
  \begin{figure}[htb]
    \centering
    \begin{equation*}
      \xymatrix@R=2px@C=2px{ 
        \ast  & \ast &\ast  & & \ast \\ 
        \rho_n \ar^{d_n}[drrrr] & \ast & & & \ast \\
        \ast  & \ast & \ast & & \M^{0,0}[\theta_1, \dots, \theta_n]^{(n)} \\
        \vdots & \vdots & \vdots & \iddots & 0  \\
        \ast & \ast &  \M^{0,0}[\theta_1, \dots, \theta_n]^{(2)}  & & 0\\
        \ast  & \M^{0,0}[\theta_1, \dots, \theta_n]^{(1)} & 0 & & 0\\
        \M^{0,0} & 0 & 0 &  & 0}
    \end{equation*}
    \caption{\small The $\Eoh_2$--page of the spectral sequence $_{II}\Eoh_*$. The notation $R^{(n)}$ denotes the $n$--th graded part of a graded ring
      $R$. The observations made in figure \ref{figPn-1} apply here as well.}
    \label{figTonAn-0}
  \end{figure}
 
 In ${}_{II}\Eoh_2$, an $m$--fold product of $\theta_i$s has tridegree $(m,m,m)$, and a nontrivial multiple
 of this by a positively graded element of $\M$ has Chow height greater than $1$. It follows for
 similar reasons to those in proposition \ref{p:AnPn} that $\rho_n$ cannot support any differential before $d_n(\rho_n)$.
 
 Since the term ${}_{II}\Eoh_n^{n,n,n}$ is the group $\M^{0,0}[\theta_1, \dots, \theta_n]^{(2n,n)}$ of homogeneous polynomials, we must have 
\[ d_n \rho_n = p(\theta_1, \dots, \theta_n) \]
 for some homogeneous polynomial $p$ of degree $n$. We also know that
 \begin{equation}
   \label{eq:pdiagonal}
   \Delta^*(p(\theta_1, \dots, \theta_n)) = p(\theta, \dots, \theta) = \theta^n.
 \end{equation}
There being no further differentials affecting this term of the spectral sequence, and no other terms of the spectral sequence converging to
$\Hoh^{2n,n}(B(T_n, \A^n \sm\{0\});R)$, we have 
\begin{equation*} \label{e:sop}
  \Hoh^{2n,n}(B(\G_m^n, \A^n\sm\{0\});R) = \frac{ \M^{0,0}[\theta_1, \dots , \theta_n]^{(n)}}{(p(\theta_1, \dots, \theta_n))}
\end{equation*}
where $p(\theta_1, \dots ,\theta_n)$ is the homogeneous polynomial of degree $n$ which we wish to determine.

We decompose $\A^n \sm \{0\}$ with the diagonal $T_n$ action into two open subschemes, $U_1 = \A^{n-1} \sm \{0\} \times \A^1$ and $U_2 =
\A^{n-1} \times \A^1 \sm \{0\}$. These open subschemes are $T_n$--invariant. By induction on $n$ and straightforward comparison, we have
\begin{align*}
  \Hoh^{*,*}(B( T_n, U_1);R)  &= \frac{\M[\theta_1, \dots, \theta_n]}{(\prod_{i=1}^{n-1}
    \theta_i)},\\
  \Hoh^{*,*}(B( T_n, U_2);R) &= \frac{\M[\theta_1, \dots, \theta_n]}{(\theta_n)}.
\end{align*}
By comparison with $B(T_n, U_1)$ and
$B(T_n, U_2)$, we deduce that $p(\theta_1, \dots, \theta_n)$ is a degree-$n$ polynomial which lies in the ideals $\left( \prod_{i=1}^{n-1}
  \theta_i\right)$ and $\left(\theta_n\right)$. The polynomial $p$ must therefore be $a\prod_{i=1}^n \theta_i$, where $a \in \M^{0,0}$.  By
reference to \eqref{eq:pdiagonal}, we see that $a=1$.
\end{proof}


We write $\sigma_i$ for the $i$--th elementary symmetric polynomial in $\Z[x_1, \dots, x_n]$, and by
extension if $\theta_1, \dots , \theta_n$ are $n$ elements in any ring, we write
$\sigma_i(\theta_j)$ to denote the $i$--th elementary symmetric polynomial in the $\theta_j$.

\begin{prop} \label{p:wasCor3}
  Suppose $n\ge 1$. Let $T_n = (\G_m)^n$ act on $\GL_n$ via the action of left multiplication by
  diagonal matrices. Then the $\Eoh_2$--page of the associated spectral sequence takes the form
  \begin{equation} \label{eq:E2pageGlnt}
    \Eoh_2^{*,*} = \frac{\EA_\M(\rho_2, \dots,\rho_n)[\theta_1, \dots,
      \theta_n]}{(\sum_{i=1}^n \theta_i)}\quad |\rho_i| = (0,2n-1,n), \, |\theta_i| = (1,1,1),
  \end{equation}
  and the differentials are generated by $d_i (\rho_i) = \sigma_i(\theta_i)$, modulo the image of previous differentials, for $2 \le i \le n$.
\end{prop}
Note: this $\Eoh_2$ page can be thought of as coming from a fictitious $\Eoh_1$--page 
\begin{equation*}
  \EA_\M(\rho_1, \dots, \rho_n)[\theta_1, \dots, \theta_n]
\end{equation*}
with nonzero differential $d_1(\rho_1) = \sum \theta_i$. In classical algebraic topology this $\Eoh_2$--page is
isomorphic to the $\Eoh_3$--page of the Serre spectral sequence of the fibration $T_n \to ET_n
\times_{T_n} \GL_n \to B\GL_n$.
\begin{proof}
  The proof proceeds by induction on $n$.  When $n=1$, $T_1 \isom \G_m \isom \Gl_1$, so there is nothing
  to prove.

  In the case $n>1$, we know that the $\Eoh_2$--page of the spectral sequence takes the anticipated form
  from corollary \ref{p:extCalc2}.

  We fix canonical maps $\GL_{n-1} \to \GL_n$ and $T_{n-1} \to T_n$, being the standard inclusions, as in proposition \ref{p:GLNsurj},
 The following diagram of group actions commutes
  \begin{equation*}
    \xymatrix{ T_{n-1} \times \Gl_{n-1} \ar[d] \ar^{\phi} [r] & T_n \times \GL_n \ar[d] \\ \Gl_{n-1} \ar^{\phi}[r] &  \GL_n  }
  \end{equation*}
  where we denote the morphisms of the group action by $\phi$, and use this to label the arrows in an abuse of notation. We obtain a
  comparison of spectral sequences, the pages of which we will denote by ${}_{III}\Eoh_*$ and ${}_{IV}\Eoh_*$. We have a map
  $\phi^*\co {}_{IV}\Eoh_* \to {}_{III}\Eoh_*$.

By reference to corollary \ref{p:extCalc2}, we may fix presentations
  \begin{align*}
    {}_{III} \Eoh_2^{*,*} & = \frac{\EA_\M(\rho'_{2}, \dots, \rho'_{n-1})[\theta'_1, \dots, \theta'_{n-1}]}{( \sum \theta'_i)} \\
    {}_{IV} \Eoh_2^{*,*} & = \frac{\EA_\M(\rho_{2}, \dots, \rho_n)[\theta_1, \dots, \theta_{n}]}{(
      \sum \theta_i)}.
  \end{align*}
  We have $\phi^*(\rho_i) = \rho'_i$, and $\phi^*(\theta_i) = \theta'_i$ for $i<n$.

  \begin{figure}[htb]
    \centering
    \begin{equation*}
      \xymatrix@R=2px@C=2px{ 
        0  & \ast &\ast  & & \ast \\ 
        \M^{0,0}\rho_n \ar^{d_n}[ddrrrr] & \ast & \ast & & \ast \\
        \vdots & \vdots & \vdots & & \vdots \\
        \ast  & \ast & \ast & & \dfrac{\M^{0,0}[\theta_1, \dots, \theta_n]^{(n)}}{(\sigma_1(\theta_1, \dots, \theta_n))} \\
        \vdots & \vdots & \vdots & \iddots \\
        \M^{0,0} \rho_2  \ar_{d_2}[drr] & \ast & \ast & \\
        0 & \ast & \dfrac{\M^{0,0}[\theta_1, \dots, \theta_n]^{(2)}}{(\sigma_1(\theta_1, \dots, \theta_n))} \\
        \M^{1,1} & \dfrac{\M^{0,0}[\theta_1, \dots, \theta_n]^{(1)}}{(\sigma_1(\theta_1, \dots, \theta_n))} \\
        0}
    \end{equation*}
    \caption{\small The $\Eoh_2$--page of the spectral sequence ${}_{IV} \Eoh_*$. In the first column, only the terms of Chow height $1$ are
      shown. The $d_n$--differential is indicated in an abuse of notation, in reality, the target of this differential is a quotient of the
      illustrated group.}
    \label{figVE}
  \end{figure}
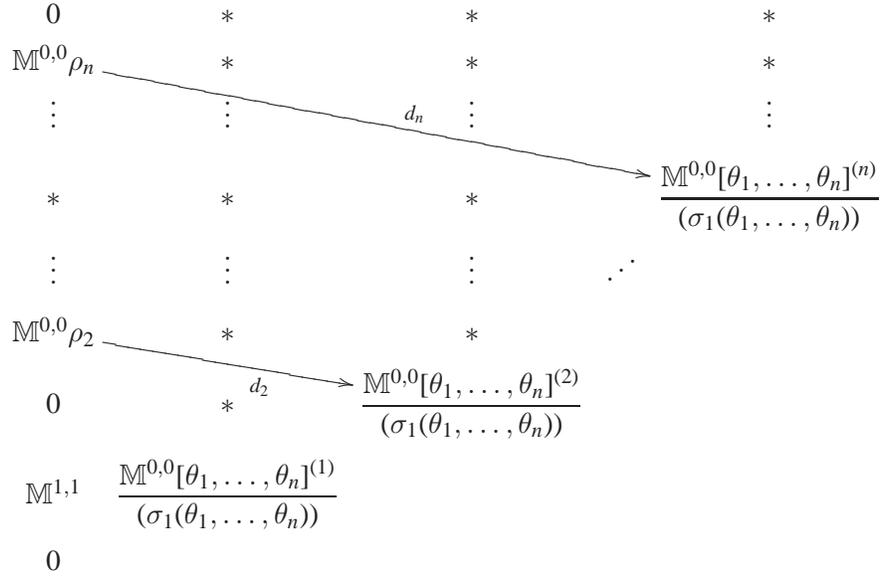

  The usual argument by total Chow height prohibits all differentials supported by $\theta_i$, and restricts the possible differentials supported by $\rho_i$ to
  \begin{equation*}
    d_i \co \M^{0,0}\rho_i \to {}_{IV}\Eoh_i^{i,i,i}
  \end{equation*}
  We determine these $d_i$--differentials when $i< n$ by comparison with ${}_{III} \Eoh_*^{*,*}$.
  
  Given a permutation $\alpha \in \Sigma_n$, the symmetric group on $n$ letters, there are maps $f_\alpha\co \GL_n \to
  \GL_n$ and $g_\alpha\co T_n \to T_n$ given by permuting the columns in the first case and permuting
  the multiplicands in the second. It is apparent that
  \begin{equation*}
    \xymatrix{ T_n \times \GL_n \ar^{g_\alpha, f_\alpha}[r] \ar[d] & T_n \times \GL_n \ar[d] \\
      \GL_n \ar^{f_\alpha}[r] & \GL_n}
  \end{equation*}
  commutes, and as a result we obtain a $\Sigma_n$--action on the spectral sequence ${}_{IV} \Eoh_*^{*,*}$. The action of $g_\alpha^*$ on
  $\Hoh^{*,*}(T_n;R)$ is to permute the generators, whereas the action of $f_\alpha^*$ is trivial, by proposition \ref{p:permutInvGln}. In
  particular this implies that the action on the $\Eoh_2$--page of the spectral sequence is to permute the classes $\theta_i$,
  but to fix those of the form $\rho_j$, so that $d_i(\rho_i) \in \Eoh_i^{i,i,i}$ is represented by a homogeneous symmetric polynomial of
  degree $i$ in the $\theta_i$.  

Suppose $i< n$, then by comparison and induction, we have
  \begin{equation*}
    \phi^*d_i(\rho_i) = d_i \rho_i' = [\sigma_i(\theta'_1, \dots,
    \theta'_{n-1})]
  \end{equation*}
  where $[\cdot]$ indicates the reduction of a class in $\M^{0,0}[\theta'_1, \dots, \theta'_n]^{(i)}$ by the ideal generated by the images
  of prior differentials and $\sum \theta'_i$. The map $\phi^*$ on $\Eoh_2$--pages is given by evaluation at $\theta_n = \rho_n=0$.
  Elementary arguments suffice to deduce that $p(\theta_1, \dots, \theta_n) = \sigma_i(\theta_1, \dots , \theta_n)$.


  As for the case of $d_n(\rho_n)$, the above-stated comparisons tell us nothing. On the other hand,
  there is a diagram of group actions
  \begin{equation*}
    \xymatrix{ T_n \times \GL_n \ar[r] \ar[d] & T_n \times \A^n \sm \{0\} \ar[d] \\ \GL_n \ar[r] &
    \A^n \sm \{0\}}
  \end{equation*}
  where the underlying map of $T_n$--spaces is simply projection on the first column, the effect of which on cohomology was computed in
  proposition \ref{p:Wincl}. It follows easily from comparison of spectral sequences with that of proposition \ref{p:wasCor2} that
  $d_n(\rho_n) = \sigma_n(\theta_i)$ in ${}_{IV}\Eoh_*^{*,*}$ as required.
\end{proof}

Write $\diag{ a_1, a_2, \dots, a_n}$ for the $n\times n$ matrix having entries $a_1, a_2, \dots, a_n$ on the main diagonal, and $0$ elsewhere.

Let $\vec{w} = (w_1, \dots, w_n) \in \Z^n$ be a set of $n$ \textit{weights}. We shall
  distinguish these from the weight-filtration of motivic cohomology by referring to the latter
  always in full. We consider the action of $\G_m$ on $\GL_n$ given on $k$--algebra-valued-points by
\begin{equation*}
  z \cdot A = \diag{z^{w_1}, z^{w_2}, \dots, z^{w_n}} A.
\end{equation*}

\begin{prop}\label{p:wasCor4}
  For the given action of $\G_m$ on $\GL_n$, there is a spectral sequence converging to the
  motivic cohomology of $B( \G_m, \GL_n)$. The $\Eoh_2$--page is given as
  \begin{align*}
    \Eoh_2^{p,q} & = \Ext^{p,q} _{\d \Hoh^{*,*}(\G_m)}(\d \Hoh^{*,*}(\GL_n, \M)) = \frac{\EA_\M(\rho_2,
    \dots, \rho_n)[\theta]}{\left( \sum_{i=1}^n w_i \theta\right)} \\
 & | \rho_i| = (0,2i-1, i),  \qquad |\theta| = (1,1,1)
  \end{align*}
  with differentials $d_i(\rho_i) = \sigma_i(\vec{w}) \theta^i$. There are no differentials supported on elements of the form $\alpha
  \theta^i$ where $\alpha \in \M^{*,*}$. The indicated differentials are the only differentials supported on classes of the form $\alpha
  \rho_i$ where $\alpha \in \M^{0,0}$.  Besides the nonzero differentials implied by the differentials described already and the product structure, if $d_i(s) \neq 0$
  for some element on the $\Eoh_i$--page, then $s$ must lie in the ideal on the $\Eoh_i$ page generated by elements of the form $\mu \rho_j$ where $j <
  i$ and where $ \sigma_j(\vec{w}) \mu = 0$.

  If $R$ is a field, then the indicated differentials and the product structure determine all the differentials.
\end{prop}
\begin{proof}
  For dimensional reasons, the classes $\theta^i$, and therefore the classes $\alpha \theta^i$ where $\alpha \in \M$, cannot support a
  nonzero differential.

  There is a map of group actions
  \begin{equation*}
    \xymatrix{ \G_m \times \GL_n \ar[d] \ar[r] & T_n \times \GL_n \ar[d] \\ \GL_n \ar@{=}[r] &
      \GL_n }
  \end{equation*}
  Where the map of groups is $z \mapsto (z^{w_1}, z^{w_2}, \dots, z^{w_n})$. The map on cohomology induced by the group homomorphism is
  \begin{align*}
    \Hoh^{*,*}(T_n) & \longrightarrow \Hoh^{*,*}(\G_m) \\
    \EA_\M(\tau_1, \dots, \tau_n) & \longrightarrow \EA_\M(\tau) \\
    \tau_i  & \mapsto w_i\tau
  \end{align*}
  This induces a map of $\Eoh_2$--pages of spectral sequences 
  \begin{align*}
    \Ext_{\d \Hoh{*,*}(T_n) }( \d \Hoh^{*,*}(\GL_n), \M) & \to \Ext_{\d \Hoh^{*,*}(\G_m)}(\d \Hoh^{*,*}(\GL_n, \M), \M),
    \\
    \frac{\EA_{\M}(\rho_2, \dots, \rho_n)[\theta_1, \dots, \theta_n]}{( \sum \theta_i)} & \to \frac{\EA_\M(\rho_2,
      \dots, \rho_n)[\theta]}{\left(\sum_{i=1}^n w_i \theta\right)}, \\
    \rho_i   \mapsto \rho_i , &\qquad |\rho_i| = (0,2i-1,i),\\
    \theta_i  \mapsto w_i \theta, &\qquad |\theta| = (1,1,1).
\end{align*}
By comparison with proposition \ref{p:wasCor3} the differentials $d_i(\rho_i)$ are already known. We have $d_i(\rho_i) =
\sigma_i(w_i)\theta^i$. Any other nonzero differential on the $\Eoh_i$--page must be supported on a term which is not in the image of the comparison
map of $\Eoh_i$--pages. Such terms are in the ideal generated by terms of the form $\mu \rho_j$ where $j < i$, and where $d_j(\mu \rho_j) = 0$, so that $\mu \rho_j$ persists
to the $\Eoh_i$--page. Since $\rho_j$ \textit{per se} cannot support any differentials beyond $d_j$ for dimensional reasons, it must be the case that
$\mu\rho_j$ is not a multiple of $\rho_j$ on the $\Eoh_i$--page, which means that $d_j(\rho_j) = \sigma_j(\vec{w})\theta^j \neq 0$. This
implies that $\sigma_j(\vec{w}) \neq 0$ in $\M^{0,0}$, but $\mu
\rho_j\neq 0$ must persist to the $\Eoh_i$--page, so $\mu \sigma_j(\vec{w}) =0$. If $R = \M^{0,0}$ is a field, then $\sigma_j(\vec{w})$ is a
  unit in $\M^{0,0}$ so this cannot happen.
\end{proof}

\begin{cor} \label{c:PGLn}
  Let $k$ be any field. Let $R$ be a field, and if $R$ is of characteristic $2$, then assume $--1$ is a square in $k$. 
  Then $\Hoh^{*,*}(\PGL_n; R)$ admits the following presentation as an $\M_R$--algebra
    \begin{equation*}
    \Hoh^{*,*}(\PGL_n; R) \isom \frac{\EA_{\M}(\rho_1, \dots, \rho_n)[\theta]}{I}\qquad |\rho_i| = (2i-1, i), \, |\theta|
    = (2,1)
  \end{equation*}
  Here $I$ is generated by $(\rho_i, \theta^i)$, where $i$ is the least integer such that $\binom{n}{i} \neq 0 $ in $R$.
  If further $R \isom \Z/p$, where $p$ is an odd prime, and $k$ is of characteristic different from $p$, then the action of the reduced power operations on $\Hoh^{*,*}(\PGL_n; \Z/p)$
  is determined by $P^i(\rho_j) = \binom{j-i}{i}\rho_{ip+j-i}$ for $ip+j-i\le n$ and $P^i(\theta) = \theta^p$ and the Cartan formula of \cite{VREDPOWER}.
\end{cor}
\begin{proof}
  If $k$ is a field such that the Beilinson-Soul\'e vanishing conjecture is known to hold for $\Spec k$, then this is a corollary of
  proposition \ref{p:wasCor4} in the case where $\vec{w} = (1,1, \dots, 1)$. In general, $\PGL_n$ gives rise to a compact object,
  $\Sigma^\infty (\PGL_n)_+$ of the stable homotopy category, \cite{DI:MCS}. Since $\PGL_n \weq B(\G_m, \GL_n)$, and since $\G_m$ and
  $\GL_n$ are both stably cellular, it follows from the same place that $\PGL_n$ is also stably cellular.
  
  Let $E$ denote either the prime-field of $k$ or the result of adjoining
  $\sqrt{-1}$ to the prime field if $R$ is of characteristic $2$. There is a convergent K\"unneth spectral sequence \[\Tor^{\Hoh^{*,*}(\Spec
    E; R)}(\Hoh^{*,*}(\PGL_{n,E}; R), \M_R) \Longrightarrow \Hoh^{*,*}(\PGL_{n, k}; R)\] Since $\Hoh^{*,*}(\PGL_{n,E}; R)$ is free as an $\Hoh^{*,*}(\Spec
  E;R)$--module, the spectral sequence is degenerate and the result follows.
  
  The results regarding the reduced power operations follow immediately from the comparison $\GL_n \to \PGL_n$ and the analogous results for $\Hoh^{*,*}(\GL_n; R)$ in \cite{WILLIAMS1}.
  We note further that the Bockstein homomorphism \[\beta\co \Hoh^{*,*}(\PGL_n; \Z/p) \to \Hoh^{*+1, *}(\PGL_n; \Z/p)\] vanishes on the classes $\theta^i$ for dimensional reasons, and that $\theta^i$, being a 
  class in $\Hoh^{*,*}(\PGL_n; \Z/p)$ which is the reduction of a class in $\Hoh^{*,*}(\PGL_n; \Z[\frac{1}{2}])$ cannot be the image of the Bockstein map. It follows that $\beta(\rho_i) = 0$ for
  all $i$, and so it follows from the Cartan formulas that the action of the Steenrod algebra $\mathcal{A}_p$ on $\Hoh^{*,*}(\PGL_n; \Z/p)$ is fully determined.
\end{proof}

In the case $p=2$, the argument above for deducing that $\beta(\rho_i) =0$ fails, because we have not calculated $\Hoh^{*,*}(\PGL_n; \Z)$. It is necessary only to know $\CH^*(\PGL_n)$, however, 
and this may be calculated using classical techniques as laid out in \cite{SEM-CHEV-TORS}. One again deduces that the classes $\theta^i \in \Hoh^{2i,i}(\PGL_n; \Z/2)$ are reductions of classes in
$\Hoh^{2i,i}(\PGL_n; \Z)$, which is all that is necessary to rule out a nonzero Bockstein map $\beta(\rho_i)$. The result of corollary \ref{c:PGLn} therefore holds for $\Z/2$ coefficients as well.

Our final result on the equivariant cohomology of $\GL_n$ deals with the equivariant cohomology for
the following $\G_m$ action on the left and on the right. If $\vec{u},
\vec{v} \in \Z^n$ are two $n$--tuples of integers, one defines an action of $\G_m$ on $\GL_n$ as
(on $k$--algebra-valued-points)
\begin{equation} \label{eq:fishFinger}
  z\cdot A =
  \diag{z^{u_1}, z^{u_2} , \dots, z^{u_n}} A
\diag{   z^{-v_1}, z^{-v_2},\dots, z^{-v_n}}.
\end{equation}

\begin{prop} \label{p:lrsss}
  Assume in this proposition that $2$ is invertible in $R$.

  For the $\G_m$ action given above, the equivariant cohomology spectral sequence has $\Eoh_2$--page:
  \begin{equation*}
    \Eoh_2^{*,*} =    \Ext_{\d \Hoh^{*,*}(\G_m;R)}( \d \Hoh^{*,*}(\GL_n;R), \M) 
  \end{equation*}
  which is an extension of $\EA_\M(\rho_1, \dots, \rho_n)[\theta]$--modules 
  \begin{equation*}
    \begin{split} \xymatrix{ 0 \ar[r] & {\displaystyle \frac{\EA_\M(\rho_2, \dots, \rho_n)[\theta]}{(\left[\sum u_i - \sum
            v_i\right]\theta)} } \ar[r] &\Ext_{\d \Hoh^{*,*}(\G_m;R)}( \d \Hoh^{*,*}(\GL_n;R), \M) \ar[r] & }  \\
\xymatrix{ \ar[r] & \Ann_{\EA_\M(\rho_2, \dots,
          \rho_n)[\theta]}(\left[ \sum u_i - \sum v_i\right] \theta) \ar[r] & 0 } \end{split}
  \end{equation*}
  If $\sum u_i - \sum v_i = 0$, then the $\Eoh_2$--page can be written as
  \begin{equation*}
    \frac{\EA_\M(\rho_1, \dots, \rho_n)[\theta]}{(\left[\sum u_i - \sum
            v_i\right]\theta)}  = \Ext_{\d \Hoh^{*,*}(\G_m;R)}( \d \Hoh^{*,*}(\GL_n; R) \M).
  \end{equation*}

  In this spectral sequence, there are differentials
  \begin{equation*}
    d_{i}(\rho_i) = \left[\sigma_i(\vec{u}) - \sigma_i(\vec{v})\right]\theta^i \pmod{ \sigma_1(\vec{u}) - \sigma_1(\vec{v}) , \dots, \sigma_{i-1}(\vec{u}) - \sigma_{i-1}(\vec{v})}
  \end{equation*}
  where $i$ is any integer between $1$ and $n$.

  These are the only differentials supported on classes of the form $a \rho_i$ where $a \in \M^{0,0}$. There are no nonzero differentials supported
  on $\theta$. Besides the nonzero differentials implied by the differentials described already and the product structure, if $d_i(s) \neq 0$ for some element on the
  $\Eoh_i$--page, then $s$ must lie in the ideal of the $\Eoh_i$--page generated by elements of the form $\mu \rho_j$ where $j < i$ and where
  $\mu \left( \sigma_j(\vec{u}) - \sigma_j(\vec{v})\right) = 0$, that is to say that $\mu \rho_j$ is in the kernel of the
  $d_j$--differentials because of the torsion of $\mu$.

  If $R$ is a field there are no further differentials other than those specified above and determined by the product-structure.
\end{prop}

The description of the $\Eoh_2$--page is understood as arising from a hypothetical
$\Eoh_1$--page that involves only the generators $\rho_1, \dots, \rho_n$ and $\theta$, with a $d_1$--differential $d_1(\rho_1) = \Big[\sum u_i - \sum v_i\Big]\theta$.

We shall also need the following ring-theoretic lemma:
\begin{lemma}
  Let $R$ be a ring in which $2$ is invertible, and let \[S =R[c_1, \dots, c_i, c_1', \dots, c_i']\] be a polynomial ring. Let $\phi\co S \to S$
  be the involution that exchanges $c_j$ and $c_j'$ for all $j$. If $f \in S$ and $f + \phi(f) = 0$, then $f$ may be written as a sum
  \begin{equation*}
    f = \sum_{j=1}^i (c_j-c_j') f_j
  \end{equation*}
  where $\phi(f_j) = f_j$.
  \end{lemma}
  \begin{proof}
   Write $d_j = \frac{1}{2}(c_j-c_j')$ and $e_j = \frac{1}{2}(c_j + c_j')$, so that \[S = R[ d_1, \dots, d_i, e_1, \dots, e_i]\] and that
      $\phi(f) = -f$ precisely when $f$ is a sum of monomial terms in which the total degree of the $d_j$ is odd. The result follows.
\end{proof}

We might also be interested in the case of $R= \Z/2$. Unfortunately, the polynomial $c_1c_1' \in \Z/2[c_1, c_1']$ furnishes a counterexample
to the lemma in this case.

  We now return to the proof of the proposition.
\begin{proof}
  First we show that the $\Eoh_2$--page is as it is described.
  
  Let $T_{2n} = (\G_m)^n \times (\G_m)^n$ act on $\GL_n$ by
  \begin{equation*}
    (a_1, \dots, a_n, b_1, \dots, b_n) \cdot A = \diag{a_1, \dots, a_n} A \diag{b_1^{-1},
        \dots, b_n^{-1}}
  \end{equation*}
  There is an evident group homomorphism $\G_m \to T_{2n}$ given by 
  \[z \mapsto (z^{u_1}, \dots, z^{u_n}, z^{v_1}, \dots, z^{v_n})\] and for this group homomorphism,
  we have a commutative map of group actions
  \begin{equation*}
    \xymatrix{ \G_m \times \GL_n \ar[r] \ar[d] & T_{2n} \times \GL_n \ar[d] \\ \GL_n \ar@{=}[r] &
      \GL_n. }
  \end{equation*}
  Suppose $\alpha, \beta \in \Sigma_n$ are each permutations on $n$--letters. Then the pair $(\alpha,
  \beta)$ acts on the group action $T_{2n} \times \GL_n \to \GL_n$, where $\alpha$ permutes the
  first $n$ terms of $T_{2n}$ and the columns of $\GL_n$, and $\beta$ permutes the last $n$ terms of $T_{2n}$, and
  the rows of $\GL_n$. We denote this action by $f_{\alpha, \beta}$. There is also an involution, which we denote $\gamma$, which acts by
  interchanging the first and last $n$ terms of $T_{2n}$ and is the map $A \mapsto A^{-1}$ on
  $\GL_n$. The identity
  \begin{equation*}\begin{split}
    \left[\diag{a_1, \dots, a_n} A \diag{b_1^{-1}, \dots , b_n^{-1}}\right]^{-1} = \\
    =  \diag{b_1, \dots, b_n} A^{-1} \diag{a_1^{-1}, \dots, a_n^{-1}} \end{split}
  \end{equation*}
  ensures that this involution is compatible with the group action.

  The action of $T_{2n}$ on $\GL_n$ yields a coaction on cohomology. We write:
  \begin{align*}
    \Hoh^{*,*}(T_{2n};R) & =  \EA_\M(\tau_1, \dots, \tau_n, \tau_1' , \dots, \tau_n'), \quad |\tau_i| =
    |\tau_i'| = (1,1) \\
    \Hoh^{*,*}(\GL_n;R) & = \EA_\M(\rho_n, \dots, \rho_1), \quad |\rho_i| = (2i-1,i).
  \end{align*}
  For dimensional reasons, the coaction must be $\tau_i \mapsto 1 \tensor \tau_i$ for $i \ge 2$, but
  the coaction $\tau_1$ is more involved, and we devote the next two paragraphs to determining it.

  For dimensional reasons we have
  \begin{equation*}
    \rho_1 \mapsto 1 \tensor \rho_1 + p_1(\tau_1, \dots, \tau_n, \tau_1', \dots, \tau_n') \tensor 1
  \end{equation*}
  where $p_1$ is a homogeneous linear polynomial.  For the inclusion $\G_m \to T_{2n}$ of the first factor, the map
  \begin{equation*}
    \xymatrix{ \G_m \times \GL_n \ar[d] \ar[r] & T_{2n} \times \GL_n \ar[d] \\ \GL_n \ar@{=}[r] &
      \GL_n }
  \end{equation*}
  is a map of group actions. Because of the naturality of the spectral sequences and by reference to
  corollary \ref{p:gmactongln}, we deduce that $p_1(\tau, 0, \dots, 0,0, \dots, 0) = \tau$. 

  On cohomology, we have $f_{\alpha,\beta}^*(\tau_i) = \tau_{\alpha^{-1}(i)}$, 
  $f_{\alpha,\beta}^*(\tau_i') = \tau_{\beta^{-1}(i)}'$ and $f_{\alpha, \beta}(\rho_i) = \rho_i$,
  the last by reference to proposition \ref{p:permutInvGln}. It follows that $p_1$ must be symmetric
  in $(\tau_1 , \dots, \tau_n)$ and $(\tau_1' , \dots, \tau_n')$ , and 
  \begin{equation*}
    p_1(\tau_1',\dots, \tau_n',  \tau_1, \dots, \tau_n) = - p(\tau_1, \dots, \tau_n, \tau_1' ,
    \dots, \tau_n')
  \end{equation*}
  Consequently $p_1(\tau_1, \dots, \tau_n, \tau_1', \dots, \tau_n') = \sum \tau_i - \sum \tau_i'$.

  By reference to corollary \ref{p:extCalc2}, we can write down the $\Eoh_2$--page of the spectral sequence for the $T_{2n}$ action on
  $\GL_n$. We denote this by ${}_{V}\Eoh_*$. The $\Eoh_2$--page is
  \begin{equation*}
    \Ext_{\d \Hoh^{*,*}(T_{2n};R)}(\d \Hoh^{*,*}(\GL_n;R), \M) =\frac{ \EA_{\M}(\rho_2, \dots, \rho_n)[
      \theta_1, \dots, \theta_n, \theta_1', \dots, \theta_n']}{ (\sum \theta_i - \sum \theta_i ')}
  \end{equation*}
  The symmetric-group actions give $f_{\alpha, \beta}^*( \rho_i) = \rho_i$, $f_{\alpha, \beta}^*(\theta_i) = \theta_{\alpha(i)}$ and
  $f_{\alpha, \beta}^*(\theta_i') = \theta_{\beta(i)}$. The involution acts as $\rho_i \mapsto - \rho_i$, $\theta_i = \theta_i'$ and
  $\theta_i' \mapsto \theta_i$.

  Now we consider which elements may support nonzero differentials, and on which pages. For dimensional reasons, powers of $\theta$ cannot
  support nonzero differentials, nor can $\rho_i$ be the image of any incoming differential. The usual arguments from Chow height and
  dimensions show that if $d_i( \mu \rho_j) \neq 0$ where $\mu \in \M^{0,0}$, then $i=j$. We will show that the differential
  $d_i(\rho_i)$ is as claimed.

We deduce from the symmetric-group actions and
  the involution that
  \begin{equation*}
    d_i(\rho_i) \equiv p_i( \theta_1, \dots, \theta_n \theta_1', \dots, \theta_n') \pmod{ (p_1, p_2, \dots, p_{i-1})}
  \end{equation*}
  where $p_i$ is symmetric in the $\theta_i$, $\theta_i'$ individually, and antisymmetric in the interchange of the two. In particular,
  writing $c_i$ for $\sigma_i(\theta_i)$ and $c_i'$ for $\sigma_i(\theta_i')$, it follows from standard results on symmetric polynomials
  that $p_i$ is a polynomial in $c_1, \dots, c_n, c_1', \dots, c_n'$

  There is map of group actions $T_n \times \GL_n$, being the action of the previous proposition, to the $T_{2n} \times \GL_n$ action at
  hand. By comparison of the spectral sequences, it follows that $p(\theta_1, \dots, \theta_n, 0, \dots, 0) \equiv \sigma_n(\theta_i)$. By
  antisymmetry we have
  \begin{equation*} \begin{split}
    p_i(\theta_1, \dots, \theta_n, \theta_1' , \dots, \theta_n') = \sigma_i(\theta_i) -
    \sigma_i(\theta_i') + \\ + q_i( \theta_1, \dots, \theta_n, \theta_1', \dots, \theta_n') \end{split}
  \end{equation*}
  where $q_i$ is of degree $i$, symmetric in $\theta_1, \dots, \theta_n$ and $\theta_1', \dots, \theta_n'$,
  antisymmetric in the interchange of the $\theta_i$ and $\theta_i'$, and $q_i$ lies in the product
  ideal \[ I = (\theta_1, \dots, \theta_n) (\theta_1', \dots, \theta_n')\]  In terms of the $c_i$, we have 
    \begin{equation*}
    p_i = c_i - c_i' + r_i(c_1, \dots, c_{i-1}, c_1', \dots, c_{i-1}').
  \end{equation*}
  where $r(c_1', \dots, c_{i-1}', c_i, \dots, c_{i-1}) = -r(c_1, \dots, c_{i-1}, c_i', \dots, c_{i-1}')$. By the lemma, $r$ lies in the
  ideal generated by $(c_1 -c_1',\dots, c_{i-1} - c_{i-1}')$. We have recursively described the differentials in the spectral sequence, they
  are
  \begin{equation*}
    d_i(\rho_i) \equiv c_i - c_i' = \sigma_i(\theta_i) -\sigma_i(\theta_i') \pmod{ \sigma_1(\theta_i)
      - \sigma_1(\theta_i'), \dots, \sigma_{i-1}(\theta_i) - \sigma_{i-1}(\theta'_i)}
  \end{equation*}

  It is now a matter of no great difficulty to use our original group homomorphism $\G_m \to T_{2n}$ to describe in full the spectral
  sequence for $\G_m$ acting on $\GL_n$. We write $\Hoh^{*,*}(\G_m;R) = \EA_\M(\tau)$. It follows by naturality that the coaction of
  $\Hoh^{*,*}(\G_m;R)$ on $\Hoh^{*,*}(\GL_n;R)$ is given by $\rho_i \mapsto 1 \tensor \rho_i$ for $i\ge 2$ and $\rho_1 \mapsto 1 \tensor \rho_1
  + \left[\sum u_i - \sum v_i \right] \tau \tensor 1$. By application of proposition \ref{p:extCalc2}, the $\Eoh_2$--page of the spectral
  sequence has the form asserted in the proposition.

  We denote this spectral sequence by ${}_{VI}\Eoh_*$. There is a comparison map of spectral sequences ${}_{V}\Eoh_* \to
  {}_{VI}\Eoh_*$, sending $\rho_i$ to $\rho_i$ for $i \ge 2$, sending $\theta_i$ to $u_i \theta$
  and $\theta_i'$ to $v_i \theta$. It follows from the comparison that in the spectral sequence ${}_{VI}\Eoh_*$, the
  differentials satisfy $d_i(\rho_i) = \left[\sigma_i(u_i) - \sigma_i(v_i)\right] \theta^i$, as claimed.

  The argument restricting the possibilities for other nonzero differentials, and eliminating such possibilities entirely when $R$ is a
  field is similar to that of the previous proposition and is omitted.
\end{proof}

\begin{cor} \label{CORLAST}
  Let $R$ be a field of characteristic different from $2$. Let $\vec{u}$ and $\vec{v}$ be two vectors in $\Z^n$, and
  let $\G_m$ act on $\GL_n$ by the action of equation \eqref{eq:fishFinger} with weights $\vec{u}$ and $\vec{v}$. Then $\Hoh^{*,*}(B( \G_m,\GL_n); R)$ admits the following presentation as an $\M_R$--algebra
  \begin{equation*}
    \Hoh^{*,*}(B( \G_m, \GL_n); R) \isom \frac{\EA_{\M}(\rho_1, \dots, \rho_{2n-1})[\theta]}{I}\qquad |\rho_i| = (2i-1, i), \, |\theta|
    = (2,1)
  \end{equation*}
  where $I$ is generated by $(\theta^j, \rho_j)$ where  $j$ is the least integer such that $\sigma_i(\vec{u}) \neq \sigma_i(\vec{v})$, or
  $I=0$ if there is no such $j$.
\end{cor}

\subsection{The Equivariant Cohomology of Stiefel Varieties}

Suppose $\vec u = (u_1, \dots, u_n)$ is an $n$--tuple of integers. Write $f_{\vec u}(z) = \prod_{i=1}^n (z- u_i)$,
a polynomial in which the coefficient of $z^i$ is $(-1)^i \sigma_i(\vec{u})$. Suppose now that $\vec v = (v_1, \dots, v_m)$ is an $m$--tuple
of integers, where $m< n$. There exist polynomials $q(z)$ and $r(z)$ in $\Q[z]$ such that $f_{\vec u} (z) = f_{\vec v}(z) q(z) + r(z)$, where
$\deg q(z) = n-m$ and $\deg r(z) < m$. In fact, since $f_{\vec v}$ is monic, the polynomials $q(z)$ and $r(z)$ have integer coefficients.

\begin{dfn}
  With notation as in the discussion above, define an \textit{approximate extension of $\vec v$ to $\vec u$} to be a vector $\apex_{\vec
    u}(\vec v)= (v_1, \dots, v_m, v'_{m+1}, \dots, v'_n) \in \C^n$ where the $v_i'$ are the roots of $q(z)$ in some order.
\end{dfn}

In general, we do not particularly care about the roots $v_i'$ themselves, rather about the elementary symmetric functions of $\apex_{\vec
  u}(\vec v)$. These agree up to sign with the coefficients of $f_{\vec v}(z) q(z)$, and are therefore integers. We highlight the following identity:
\begin{lemma} \label{l:tiny}
  With notation as above, $\sigma_i(\vec u) - \sigma_i(\apex_{\vec u}(\vec u))$ is $(-1)^{n-i}$ times the coefficient of $z^{n-i}$ in $r(z)$.
\end{lemma}

\begin{theorem} \label{th:EqCohStiefel}
  Suppose the Beilinson--Soul\'e vanishing conjecture holds for the spectrum of the ground field, $\Spec k$. Let $R = \Z[\frac{1}{2}]$. 

  Let $\vec{u} =(u_1, \dots, u_n) \in \Z^n$ and $\vec{v} = (v_1, \dots, v_m) \in \Z^m$ be two sequences of weights with $m<n$. Consider the $\G_m$--action on $V_m(\A^n)$ given by
  \begin{equation*}
    z \cdot A = \diag{u_1, \dots, u_n} A \diag{v_1^{-1}, \dots, v_m^{-1}}.
  \end{equation*}

  The spectral sequence computing $\Hoh^{*,*}(B(\G_m, V_m(\A^n)); R)$ for the given $\G_m$--action has $\Eoh_2$--page
  \begin{equation*}
    {}_I\Eoh_2 = \Ext_{\d \Hoh^{*,*}(\G_m;R)}(\d \Hoh^{*,*}( \GL_n;R), \M) = \EA_\M(\rho_{n-m+1}, \dots,  \rho_n)[\theta].
  \end{equation*}
  Let $n-m+1 \le k \le n$ and suppose that 
  \begin{equation*}
    d_{j}(\rho_j) = 0 \quad \text{ for $n-m+1 \le i < k$},
  \end{equation*}
  then we have, in the given spectral sequence,
  \begin{equation} \label{eq:modprel}
    d_{k}(\rho_k) =  \big[\sigma_k(\vec{u}) - \sigma_k(\apex_{\vec{u}}(\vec v))\big] \theta^k.
  \end{equation}
\end{theorem}
\begin{proof}
  Strictly speaking, the proof proceeds by induction on $k\ge n-m+1$, although most of the
  difficulty is already evident in the base case of $k = n-m+1$. The arguments for the case $k = n-m +1$ and
  for the induction step are very similar; we shall give both in parallel as much as possible.

  Since the cohomology of $V_m(\A^n)$ is $\EA_{\M}( \rho_{n-m+1}, \dots, \rho_n)$,  for which see corollary \ref{p:Wincl}, it follows from \ref{p:extCalc} that the
  $\Eoh_2$--page is as claimed.
  
  For dimensional reasons powers of $\theta$ cannot support a nonzero differential. From this, the product structure, and the usual argument
  from Chow height, it follows that the first nonvanishing differential must be of the form $d_i(\rho_i) = a \theta^i$ for some $i$, where
  $a \in \Z[\frac{1}{2}]$. Suppose $d_i(\rho_i) =0$ for $n-m+1 \le i < k$, a condition that is vacuously satisfied if $k = n-m+1$. We then
  have $d_k(\rho_k) = a \theta^k$, and it remains to determine $a \in \Z[\frac{1}{2}]$.

  Let $f_{\vec u}(z)$, $f_{\vec v}(z)$ and $q(z)$ be as in the discussion immediately preceding the theorem. By a corollary of the Frobenius
  density theorem, the polynomial $q(z)$ splits over $\Z/p$ for infinitely many odd primes $p$. Let $P$ denote the set of all such primes.

  To identify $a \in \Z[\frac{1}{2}]$ it suffices to calculate the class of $a$ in $\Z/p$ for infinitely many primes $p$, that is, it
  suffices to verify equation \eqref{eq:modprel} modulo infinitely many primes $p$. We shall verify it for all primes $p \in P$.

  Fix a particular prime $p \in P$. By construction there are elements $\bar v'_{m+1}, \dots, \bar v'_n$ of $\Z/p$ which
  are roots of $q(z) \in \frac{\Z}{p}[z]$, and therefore there are integers $v'_{m+1}, \dots, v'_n$ whose reductions are the $\bar
  v'_i$.  Let $\vec v'$ denote the concatenation of $\vec v$ with these integers, viz.~$\vec v' = (v_1, \dots, v_m, v_{m+1}', \dots, v_n')$.

  Equip $\GL_n$ with the $\G_m$--action given by weights $\vec u$ on the left and $\vec v'$ on the right. Then projection onto the first $m$
  columns $\GL_n \to V_m(\A^n)$ is $\G_m$--equivariant. 

  When we consider the spectral sequence computing $\Hoh^{*,*}(B(\G_m, \GL_n); \Z/p)$ for this action, as in proposition \ref{p:lrsss}, we
  find that $d_i(\rho_i) = \left[\sigma_i(\vec u) - \sigma_i(\vec v')\right] \theta^i$. But $\sigma_i(\vec u)$ is $(-1)^i$ times the
  coefficient of $z^{n-i}$ in $f_{\vec u}(z)$ and $\sigma_i(\vec v')$ is $(-1)^i$ times the coefficient of $z^{n-i}$ in $f_{\vec v}(z)q(z)$,
  and since $f_{\vec u}(z) = f_{\vec v}(z) q(z) + r(z)$ with $\deg r(z) < m$, it follows in particular that if $i < n-m$, then $d_i(\rho_i)
  = 0$.

  We continue to work with $\Z/p$--coefficients.  For $i \ge n-m$, the class $\rho_i$ appears in both the spectral sequence computing
  $\Hoh^{*,*}(B(\G_m, \GL_n); \Z/p)$ and the sequence computing $\Hoh^{*,*}(B(\G_m, V_m(\A^n)); \Z/p)$. By comparison, if $d_i(\rho_i) = 0$ for $n-m \le i < k$ in
  the latter, then $d_i(\rho_i) = 0$ for $n-m \le i < k$ in the former. We find $d_k(\rho_k) = \left[\sigma_k(\vec u) -
    \sigma_k(\vec v') \right] \theta^k$. Observe that $\left[ \sigma_k(\vec u) - \sigma_k(\vec v')\right]$ is $(-1)^{n-k}$ times the
  coefficient of $z^{n-k}$ in $f_{\vec u}(z) - f_{\vec v}(z) q(z) = r(z)$, but this agrees with the reduction to $\Z/p$ of $[\sigma_k(\vec
  u) - \sigma_k(\apex_{\vec u}(\vec v))]$ by lemma \ref{l:tiny}. In particular, we have established equation \eqref{eq:modprel} modulo $p$.
\end{proof}

Unfortunately, this method of proof establishes only the first nonzero differential of the form
$d_{k}(\rho_k) = C\theta^k$; we cannot push it further to describe the subsequent
differentials. We conjecture that the pattern established in the theorem continues, that the
differential takes the form
\begin{equation*}
  d_{k}(\rho_k) = \big[\sigma_k(\vec{u}) - \sigma_k(\vec \apex_{\vec u}(\vec v)) \big] \theta^k
\end{equation*}
modulo the appropriate indeterminacy for all $k$.

\appendix
\section{Homological Algebra}
The material in this appendix is provided to support with proof the general assertion that the
spectral sequences we calculate carry the expected product-structure. If the base ring were a
field then all the following results would be standard, but we need them in the
case where the base-ring is the coefficient ring $\M = \Hoh^{*,*}( \pt; R)$, where they remain true
provided one considers only finitely generated free modules, as we do.

The homological algebra we need deals with Hopf algebras over bigraded rings, graded-commutative in the first grading, commutative in the
second. As a convention we fix such a ring, $\M$. All modules will be finitely-generated over $\M$. For an $\M$--module $A$ we shall write
$\d{A}$ for $\Hom_\M(A, \M)$. If $A$ is free, and equipped with a distinguished generating set $\{a_1, \dots, a_n\}$, we will write
$\{\d{a}_1, \dots, \d{a}_n\}$ for the dual generating set of $\d{A}$. Since we deal very often with exterior algebras, $\EA_\M(a_1, \dots,
a_n)$, we will decree here that such an algebra, when understood as an $\M$--module, should be equipped with a distinguished basis 
consisting of the nonzero products of the $a_i$.

We shall need the notion of \textit{relative\/} $\Ext$--groups for the $\M$--algebra $\d \Hoh^{*,*}(G)$. These can be defined via a bar construction as
follows. let $N,M$ be $S$--modules, then we can form the bar complex $\bot_* N$ whose
$p$--th term is $N \tensor_\M S^{\tensor p}$, so that
\begin{equation*}
  \Ext^*_{S/\M}(N,M) = H_*( \Hom_S(\bot_* N, M))
\end{equation*}
For the properties of such groups we refer to \cite[Chapter 8]{WEIBEL}. In particular we will need the following propositions.

\begin{prop} \label{p:weibelone}
  Let $R$ be an $\M$--algebra, and let 
  \begin{equation*}
    \xymatrix{ 0 \ar[r] & M_1 \ar[r] & M_2 \ar[r] & M_3 \ar[r] & 0 }
  \end{equation*}
  be a short exact sequence of $R$--modules that is split when considered as a sequence of
  $\M$--modules. Let $N$ be an $R$--module. Then there is a long exact sequence of groups
{\small  \begin{equation*} 
    \xymatrix@C=12pt{ \ar[r] &\Ext^*_{R/\M}(M_3, N) \ar[r] & \Ext^*_{R/\M}(M_2, N) \ar[r] & \Ext^*_{R/\M}(M_1, N)
      \ar^{\partial}[r] & \Ext_{R/\M}^{*+1}(M_3, N) \ar[r] &.}
  \end{equation*}}
\end{prop}

\begin{prop} \label{p:weibeltwo}
  Let $R_1$, $R_2$ be $\M$--algebras, with $\M$ graded-commutative and let $M_i, N_i$ be $R_i$--modules
  for $i=1,2$. Then $M_1 \tensor_\M M_2$, $N_1 \tensor_\M N_2$ are $R_1 \tensor_\M R_2$--modules, there
  is an external product
  \begin{equation*}
\begin{split}    \Ext^*_{R_1/\M}(M_1, N_1) \tensor_\M \Ext^*_{R_2/\M}(M_2, N_2) \to \\ \to \Ext^*_{R_1 \tensor_\M R_2 /\M}(M_1
    \tensor_\M M_2, N_1 \tensor_\M N_2) \end{split}
  \end{equation*}
  which is natural in all four variables and commutes with the connecting homomorphism of
  proposition \ref{p:weibelone}.
\end{prop}

The product arises from a standard Alexander-Whitney construction on $\bot_* M_1 \tensor \bot_* M_2$.

In general we shall be dealing with group actions in the category of $\Spaces$, which is to say a
group object $G$, an object $Y$, and a map $G \times Y \to Y$.  The ring $\d S= \d \Hoh^{*,*}(G;R)$ is therefore a Hopf algebra over $\M$,
so there is an algebra homomorphism $\d S \to \d S \tensor_\M \d S$. Write $\d N$ for $\d \Hoh^{*,*}(Y;R)$, since we
shall be treating of $\Ext^*_{\d S}(\d N, \M)$, and there is a map:
\begin{equation*}
 \d N \to \d N \tensor_\M \d N,
\end{equation*}
arising from the diagonal $Y \to Y \times Y$. There is a $G$--action on $Y \times Y$, \textit{via} the diagonal $ G \to G \times G$. The
coalgebra map $\d N \to \d N \tensor_\M \d N$ is therefore $\d S$--linear. These conditions make $\d N$ into an $\d S$--module-coalgebra. For such data, there is a product:
\begin{equation*}
  \xymatrix{ \Ext^*_{\d S/\M}(\d N, \M) \tensor_\M \Ext^*_{\d S/\M}(\d N, \M)
  \ar[r] &  \Ext^*_{\d S \tensor \d S/\M}( \d N \tensor_\M \d N, \M \tensor_\M \M) \ar[dl]  \\    \Ext^*_{\d S/\M}(\d N\tensor_\M \d N, \M
  \tensor_\M \M) \ar[r] & \Ext^*_{\d S/\M}(\d N, \M), }
\end{equation*}
where the first map is the external product, and the other maps are those arising from the
functoriality of $\Ext^*_{\d S/\M}(\d N,M)$.

We synopsize:
\begin{prop} \label{p:extprod}
  Suppose $\d S$ is a Hopf algebra over $\M$, $\d N$ is a module-coalgebra over $\d S$, then there is a natural 
  ring structure on $\Ext^*_{\d S/\M}(\d N, \M)$.
\end{prop}
\begin{prop} \label{p:extmod}
  Let $\d S$ be a Hopf algebra over $\M$, let \[\xymatrix{ 0 \ar[r] & \d N_1 \ar[r] & \d N_2 \ar[r] &
    \d N_3 \ar[r] & 0}\] be a short exact sequence of $\d S$--modules that splits as a sequence of
  $\M$--modules. Then the long exact sequence of relative $\Ext$--groups
  \begin{equation*}
    \xymatrix{ \Ext^*_{\d S/\M}(\d N_3, \M) \ar[rr] & &  \Ext^*_{\d S/\M}(\d N_2, \M) \ar[dl] \\ &
       \Ext^*_{\d S/\M}(\d N_1, \M) \ar^{\partial}[ul] }
  \end{equation*}
  is in fact a long exact sequence of $\Ext^*_{\d S/\M}(\M,\M)$--modules.
\end{prop}
\begin{proof}
  There is a (trivial) map $\d N_i \to \M \tensor_\M \d N_i$. Both $\d N_i$ and $\M$ are $\d
  S$--modules, and consequently we can use the external product on relative $\Ext$--groups as before to
  obtain a product $\Ext^*_{\d S/\M}(\M,\M) \tensor_\M \Ext^*_{\d S/\M}(\d N_i, \M)$. By proposition
  \ref{p:weibeltwo} the $\Ext^*_{\d S/\M}(\M,\M)$--action is compatible with the long exact sequence of
  proposition \ref{p:weibelone}.
\end{proof}

\begin{prop}
  Let $\d S$ be a Hopf algebra over $\M$. Let $\d N$ be an $\d S$--module--coalgebra that is free as an $\M$--module. The ordinary
  $\Ext$--group $\Ext^*_{\d S}(\d N, \M)$ agrees with the relative $\Ext^*_{\d S/\M}(\d N, \M)$. In the case of $\d N = \M$, the product
  coincides with the usual Yoneda product on  $\Ext^*_{\d S}( \M , \M)$.
\begin{proof}

  The first assertion follows since $\bot_* \d N$ is a free resolution of $\d N$ when $\d N$ is free over $\M$.


  Let $S$ be an $\M$--algebra . Let $D \d S$ be the derived category of bounded-below complexes of $\d S$--modules. Let $A,B$ be $\d S$--modules. Recall
  that $\Ext^*_{\d S}(A,B) = \Hom_{D{\d S}}(A,B)$.   We observe that for two $\d S$--modules, $A$, $B$, the product $A \tensor_\M B$ is an $\d S \tensor_\M \d S$--module and consequently an $\d
  S$--module by restriction of scalars. It follows that $\tensor_\M$ is a right-exact bifunctor on the category of $\d S$--modules. We denote
  the derived version of this functor also by $\tensor^L_\M$. It provides us with a monoidal structure on $D \d S$.
  
  The product on $\Ext^*_{\d S/\M}(\d N, \M)$ may also be constructed as the product that takes two maps $f, g:\d N \to \M$ in the derived
  category to
  \[\d N \to \d N \tensor_\M \d N \weq \d N \tensor_\M^L \d N \overset{f \tensor g} \longrightarrow \M \tensor_\M \M \isom \M .\]
  
  In the specific case of $\d N = \M$, we have also a Yoneda product, which in the derived category is the composition $g \circ f\co \M
  \to \M$. The proof that these two products coincide is standard, and is known as the Eckmann-Hilton argument.
\end{proof}
\end{prop}

\begin{prop}
  Let $\M$ be a graded-commutative algebra, and let 
\[ S= \EA_\M( \alpha_1, \dots, \alpha_n, \beta_1, \dots, \beta_m) \] 
be a Hopf algebra over $\M$, with exterior algebra structure, grading given by $|\alpha_i| = a_i$ (the grading on the $\beta_i$ is immaterial) and coalgebra
  structure given by the stipulation that $\alpha_i$, $\beta_i$ are all primitive. Write
 \[\d N = \d S/ (\d \beta_1, \dots, \d \beta_m) = \EA_\M(\d
  \alpha_1, \dots, \d \alpha_n)\] 
which inherits a $\d S$--linear coproduct map $\d N \to \d N \tensor_\M \d N$, the action of $\d S$ on $\d N \tensor_\M \d N$ being
  via $\d S \to \d S \tensor_\M \d S$. There is an isomorphism of bigraded $\M$--algebras:
  \begin{equation*}
    \Ext^*_{\d S}( \d N , \M) \isom \M[\theta_1, \dots, \theta_m],
  \end{equation*}
  with $|\theta_i| = (1,a_i)$, and this isomorphism is natural in $\d S$, $\d N$ and $\M$.
\end{prop}
\begin{proof}
  All tensor products are taken over $\M$. The naturality of the isomorphism follows from the
  naturality of all constructions carried out below, and we shall not mention it again.

In the case $n=0$ we have $N = \M$ and the result \[\Ext_{\d S}(\M , \M) = \M
  [\theta_1, \dots, \theta_m]\] is well-known.

  To compute the general case of $\Ext^*_{\d S}(\d N , \M)$, we resolve $\d N$ by a standard
  resolution
  \begin{equation*}
    \xymatrix{  \ar[r] &  \bigoplus_{1 \le i,j \le n} \d S \beta_i\beta_j \ar[r] & \bigoplus_{1 \le i \le n}\d S\beta_i \ar[r] & \d S
      \ar[r] & \d N.}
  \end{equation*}
  If we write $F_j$ for the $j$--th term; the image of the differentials lie in the submodules $(\beta_1, \dots, \beta_n)F_j$. In particular,
  application of $\Hom_{\d S}( \cdot, \d N)$ yields a complex with trivial differential. It follows that application of $\Hom_{\d S}( \cdot
  , \M)$ coincides with the result of applying $\Hom_{\d S}( \cdot, \d N)$, followed by $\cdot \tensor_{\d N} \M$. Consequently the map
  $\Ext^*_{\d S}(\d N , \d N) \to \Ext^*_{\d S}(\d N, \M)$ of $\M$--modules is the map
  \begin{equation*}
    \d  N[\theta_1, \dots, \theta_m] \to \d  N[\theta_1, \dots, \theta_m] \tensor_{\d N} \M \isomto \M[\theta_1, \dots, \theta_m]
  \end{equation*}
  Since this is a ring map by the naturality of the ring structure on $\Ext$, see \cite[Chapter 8]{WEIBEL}, the result follows.
\end{proof}

\begin{prop} \label{p:extCalc}
  Let $\M$ be a bigraded ring, let $A$ be an exterior algebra \[A=\EA_{\M}(\alpha_1, \dots, \alpha_n,
   \beta_1, \dots, \beta_m)\] that is also a Hopf algebra with the elements $\alpha_i, \beta_i$
  primitive and homogeneous. Let $B$ be a ring \[B =\EA_R(\alpha_1', \dots, \alpha_n', \gamma_1 ,\dots,
  \gamma_p)\] with $\alpha_i'$, $\gamma_i$ homogeneous which is equipped with a comodule structure
  over $A$, where $B \to A \tensor_\M B$ is given by
   \begin{equation*}
     \alpha_i' \mapsto \alpha_i \tensor 1 + 1 \tensor \alpha'_i, \quad \gamma_i \mapsto 1 \tensor \gamma_i.
   \end{equation*}
   Let $\d A$ denote the dual algebra of $A$ over $\M$, and $\d B$ the dual coalgebra over $\M$. Then
   $\d B$ is an $\d A$--module, and there is an isomorphism:
   \begin{equation*}
     \Ext^*_{\d A}(\d B, \M) = \EA_\M(\gamma'_1, \dots, \gamma'_p)[\theta_1, \dots, \theta_m],
   \end{equation*}
   which is again natural in $\d A, \d B, \M$. There is a map:
   \begin{equation*}
     \Ext^*_{\d A}(\d B, \M) \to \Ext^*_\M(\d B, \M) = B,
   \end{equation*}
   mapping $\gamma_i'$ to $\gamma_i$. The element $\theta_i$ corresponds to $\beta_i$; if $\beta_i$
   has bidegree $(r,s)$, then $\theta_i$ has bidegree $(r,s)$ in $\Ext^1_{\d A}(\d B, \M)$.
\end{prop}
\begin{proof}
  All constructions and isomorphisms are natural.
  Since $B$ is free and finitely generated as an $\M$--module, it follows that $\Hom_\M(B,\M)
  = \d B$ is too. Consequently: \[\Ext^*_\M(\d B, \M) = \Hom_\M( \d B, \M) = \d{\d B} = B.\]

  For a subset $J \subset \{ 1, \dots, p\}$, write $\gamma_J$ for the product $\prod_{i \in J} \gamma_i$. We can
  decompose $\d B$ as a direct sum indexed over products 
  \begin{equation*}
    \bigoplus_{J \subset \{1, \dots, p\}} \frac{\d A}{(\d \beta_1, \dots, \d \beta_m)}\d \gamma_J
  \end{equation*}
  With the given decomposition and by use of the previous proposition, we have
   \begin{equation*}
    \Ext^*_{\d A}(\d B, \M) =  \bigoplus_{J \subset \{1, \dots, p\}} \Ext^*_{\d A}\left( \frac{\d A}{(\d
        \beta_1, \dots, \d \beta_m)}, \M \right) \d {\d\gamma}_J = \bigoplus_{J \subset \{1,\dots, p\}}
    \M[\theta_1, \dots, \theta_m] \d{ \d \gamma}_J
  \end{equation*}
  The indeterminates $\theta_i$ lie in $\Ext^1_{\d A}(\d B, \M)$. What remains to be determined is the multiplication $\d {\d \gamma}_J
  \d{\d\gamma}_{J'}$, but the ring map $\Ext^*_{\d A}(\d B, \M) \to B$ takes $\d {\d\gamma}_i$ in the former to $\gamma_i$ in the latter, and
  since
  \begin{equation*}
    \d {\d\gamma}_i \in \Ext^0_{\d A}(\d B, \M) = \bigoplus_{J \subset \{1, \dots, p\}}
  \M \d {\d \gamma}_J
  \end{equation*}
  it follows that $\d {\d\gamma}_i \d {\d\gamma}_j = \widehat{\widehat{\gamma_i
      \gamma_j}}$. We are therefore justified in dropping the distinction and write $\d{\d\gamma}_i
  = \gamma_i$.
\end{proof}

We must account for one additional complexity in our calculation of $\Ext$--rings.
\begin{prop} \label{p:extCalc2}
  Let $\M$ be a bigraded ring, let $A$ be an exterior algebra \[A=\EA_{\M}(\alpha_1, \dots, \alpha_n,
  \beta_1, \dots, \beta_m)\] that is also a Hopf algebra where the elements $\alpha_i, \beta_i$ are
  primitive and homogeneous. Let $B$ be a ring 
\[B = \EA_\M(\alpha_1', \dots, \alpha_n', \gamma_1 ,\dots,  \gamma_p, \eta)\]
 where the elements $\alpha_i'$, $\gamma_i$ and $\eta$ are homogeneous, and which is equipped with a comodule structure
  over $A$, where $B \to A \tensor_\M B$ is given by
   \begin{equation*}
     \alpha_i' \mapsto \alpha_i \tensor 1 + 1 \tensor \alpha'_i, \quad \gamma_i \mapsto 1 \tensor \gamma_i, \quad \eta \mapsto 1 \tensor
     \eta +  \sum_{i=1}^m b_i \beta_i \tensor 1
   \end{equation*}
   where $b_1, \dots, b_m \in \M^{0,0}$.  Let $\d A$ denote the dual algebra of $A$ over $\M$, and $\d B$ the dual coalgebra over $\M$. Then
   $\d B$ is an $\d A$--module, and there is an exact sequence of trigraded $\EA_\M(\gamma_1',\dots, \gamma_p')[\theta_1, \dots, \theta_m]$--modules
   \begin{equation*}\begin{split}
       0 \to \frac{\EA_\M(\gamma'_1, \dots, \gamma'_p)[\theta_1, \dots,
       \theta_m]}{\left(\sum b_i \theta_i \right)} \to \\ \to\Ext^*_{\d A}(\d B, \M)  \to \Ann_{\EA_\M(\gamma_1', \dots, \gamma_p')[\theta_1,
       \dots, \theta_m]}\left(\sum b_i \theta_i\right) \to 0 \end{split}
   \end{equation*}

   There is a natural map
   \begin{equation*}
     \Ext^*_{\d A}(\d B, \M) \to \Ext^*_\M(\d B, \M) = B
   \end{equation*}
   mapping $\gamma_i'$ to $\gamma_i$. The element $\theta_i$ corresponds to $\beta_i$; if $\beta_i$
   has bidegree $(r,s)$, then $\theta_i$ has bidegree $(r,s)$ in $\Ext^1_{\d A}(\d B, \M)$.
\end{prop}
\begin{proof}
  As before, we can reduce the problem by decomposing the $\d A$--module $\d B$ into direct summands generated by monomials in the $\d
  \gamma_i$. We therefore assume that $p=0$. 

  We define the ring $\d S= \EA_\M(\d \alpha_1, \dots, \d \alpha_m)$, which is an $\d A$--algebra. There is an isomorphism $\d B = \d S \oplus \d S \d
  \eta$, where the $\d S$--module-structure is the evident one and $\d \beta_i \cdot 1 = b_i \d \eta$. 

  The map $\M \to \d S$ induces the evident map \[\M[\theta_1, \dots, \theta_m] \isom \Ext_{\d A}(\d S, \M) \to \Ext_{\d A}(\M,
  \M) = \M[\theta_1, \dots, \theta_m, \phi_1, \dots, \phi_n],\]where the $\theta_i$ correspond to the $\beta_i$ and the $\phi_i$ to the $\alpha_i$.

  There is a short exact sequence of $\d A$--modules which splits as a sequence of $\d S$--modules
  \begin{equation*}
    \xymatrix{ 0 \ar[r] & \d S \ar^{1 \mapsto \d \eta}[rr] & & \d B \ar[r] & \d S \ar[r] & 0}
  \end{equation*}
  From this one obtains a long exact sequence of $\Ext$--groups which is a sequence of $\Ext^*_{\d A}(\M, \M)$--modules, and so of $\Ext^*_{\d
    A}(\d S, \M)$--modules:
  \begin{equation*}
    \xymatrix{ \Ext^*_{\d A}(\d S,\M) \ar[rr] &  & \Ext^*_{\d A}(\d B,\M) \ar[dl] \\ & \Ext^*_{\d A}(\d S, \M)
      \ar^{\partial}[ul] & }
  \end{equation*}
  By explicit calculation in the snake lemma, the boundary map 
  \[\Hom_{\d A}(\d S, \M) = \M \to \Ext^1_{\d A}(\d S,\M) = \M[\theta_1,\dots, \theta_m]^{(1)}\]
  takes $1$ to $\sum b_i \theta_i$. The long exact sequence now gives the exact sequence
  \begin{equation*}
    \xymatrix@C=8px{ \ar[r] &   \ar^\isom[d] \Ext^*_{\d A}(\d S,\M) \ar^{1 \mapsto \sum
        b_i \theta_i }[rrr] & & &   \ar^\isom[d] \Ext^*_{\d A}(\d S, \M)  \ar[r] &
      \Ext^*_{\d A}( \d B ,\M) \ar[r] &  \Ext^*_{\d A}(\d S,\M) \ar[r] & \\
       &\M[\theta_1, \dots, \theta_m] & & &  \M[\theta_1, \dots, \theta_m]  }
  \end{equation*}
  and this is in fact an exact sequence of $\Ext^*_{\d A}(\d S,\M)$--modules, by proposition \ref{p:extmod},
  so $\Ext^*_{\d A}(\d B ,\M)$ fits into a short exact sequence:
  \begin{equation*}
    \xymatrix{ 0 \ar[r] & {\displaystyle \frac{\M[\theta_1, \dots, \theta_m]}{(\sum b_i \theta_i)} }\ar[r] & \Ext^*_{\d A}(\d B,\M)  \ar[r] & \Ann_{\M[\theta_1,
        \dots, \theta_m]}( \sum b_i \theta_i) \ar[r] & 0.}
  \end{equation*}
  This is the claim of the proposition in the case $p=0$.
\end{proof}

\begin{cor} \label{c:brelprime}
  In the notation of the preceding problem, if each $b_i$ is invertible in $\M^{0,0}$, then the $\Eoh_2$--page takes the form
  \begin{equation*}
    \frac{\EA_\M(\gamma'_1, \dots, \gamma'_p)[\theta_1, \dots, \theta_m]}{\left(\sum b_i \theta_i \right)} \isom \Ext^*_{\d A}(\d B, R).
  \end{equation*}
\end{cor}
\begin{proof}
  Under the hypotheses given, the element $\sum b_i \theta_i$ is not a zerodivisor in the ring $\EA_\M(\gamma_1' , \dots, \gamma_p')
  [\theta_1, \dots, \theta_m]$.
\end{proof}

\section*{Acknowledgements}
This paper has been too long in the writing, and the author is certain that many people have helped him in ways that he has gracelessly
forgotten. He wishes to thank in particular Gunnar Carlsson for introducing him to motivic cohomology and the fiber-to-base spectral
sequence; he also wishes to thank Eric Friedlander for telling him the folklore of the Beilinson-Soul\'e vanishing conjecture and Anssi
Lahtinen for many stimulating conversations.

%
%

\bibliography{BThesis}

%
\bibliographystyle{gtart}

\end{document}